\documentclass[12pt]{amsart}

\usepackage{graphicx}
\usepackage{pst-all}
\usepackage{amsmath}
\usepackage{amsfonts}
\usepackage{amssymb}
\usepackage{amsthm}
\usepackage{url}
\usepackage{subfigure}
\input xy
\xyoption{all}
\usepackage{nicefrac}

\theoremstyle{theorem}
\newtheorem{theorem}{Theorem}
\newtheorem{corollary}[theorem]{Corollary}
\newtheorem{lemma}[theorem]{Lemma}
\newtheorem{proposition}[theorem]{Proposition}
\theoremstyle{definition}
\newtheorem{definition}[theorem]{Definition}
\newtheorem{example}[theorem]{Example}
\newtheorem{remark}[theorem]{Remark}

\thanks{The authors are partially supported by the Spanish Ministerio de Econom{\'{i}}a y Competitividad (grant MTM2015-63612-P)}

\subjclass[2010]{}

\begin{document}

\author{H. Barge}
\address{E.T.S. Ingenieros inform\'{a}ticos. Universidad Polit\'{e}cnica de Madrid. 28660 Madrid (Espa{\~{n}}a)}
\email{h.barge@upm.es}

\author{J. J. S\'anchez-Gabites}
\address{Facultad de Ciencias Econ{\'{o}}micas y Empresariales. Universidad Aut{\'{o}}noma de Madrid, Campus Universitario de Cantoblanco. 28049 Madrid (Espa{\~{n}}a)}
\email{JaimeJ.Sanchez@uam.es}

\title{Knots and solenoids that cannot be attractors of self-homeomorphisms of $\mathbb{R}^3$}
\begin{abstract} As a first step to understand how complicated attractors for dynamical systems can be, one may consider the following realizability problem: given a continuum $K \subseteq \mathbb{R}^3$, decide when $K$ can be realized as an attractor for a homeomorphism of $\mathbb{R}^3$. In this paper we introduce toroidal sets as those continua $K \subseteq \mathbb{R}^3$ that have a neighbourhood basis comprised of solid tori and, generalizing the classical notion of genus of a knot, give a natural definition of the genus of toroidal sets and study some of its properties. Using these tools we exhibit knots and solenoids for which the answer to the realizability problem stated above is negative.
\end{abstract}

\maketitle

\section*{Introduction}

When studying dynamical systems from a qualitative perspective it is usual to deal with \emph{stable attractors}. A (Lyapunov) stable attractor is a compact invariant set $K$ such that the trajectory of every point sufficiently close to $K$ approaches $K$ asymptotically. The stability condition means that points that are initially close to the attractor do not wander away too much before approaching the attractor. Heuristically, in the long run the dynamics of the system in a neighbourhood of the attractor will be indistinguishable from the dynamics on the attractor itself and so the latter captures, locally, the long term behaviour of the system.

Attractors frequently exhibit a very complicated topological structure, and it is natural to wonder how complicated this can be. This may be substantiated in the following ``realizability problem'': find criteria that, given a compactum $K \subseteq \mathbb{R}^n$, decide whether there there exists a dynamical system on $\mathbb{R}^n$ for which $K$ is an attractor. This is not completely precise yet, for one could consider continuous or discrete dynamics, require $K$ to be a global or just a local attractor, etc. Several authors have obtained results about realizability problems of this sort, and as an illustration we may refer the reader to \cite{bhatiaszego1}, \cite{garay1}, \cite{gunthersegal1}, \cite{jimenezllibre1}, \cite{peraltajimenez1}, \cite{mio5}, \cite{sanjurjo1} for the continuous case and to \cite{crovisierrams1}, \cite{duvallhusch1}, \cite{gunther1}, \cite{kato1}, \cite{ortegayo1}, \cite{mio6}, \cite{mio7} for the discrete case. A general idea that emerges from these works is that it is not the topology of $K$ what plays a crucial role in the realizability problem, but rather how $K$ sits in $\mathbb{R}^n$. For instance, any closed orientable surface can be clearly embedded in $\mathbb{R}^3$ as an attractor, but it can also be embedded in such a way that it cannot be so realized. The main goal of this paper is to exhibit more examples of this phenomenon, and for this purpose we shall focus on a class of compacta that we call \emph{toroidal}.


We recall that a compactum $K \subseteq \mathbb{R}^n$ is called \emph{cellular} if it has a neighbourhood basis comprised of $n$--cells; that is, sets homeomorphic to the closed unit ball in $\mathbb{R}^n$. A result of Garay \cite{garay1} states that both in the continuous and the discrete case a compactum $K \subseteq \mathbb{R}^n$ can be realized as a \emph{global} attractor if and only if it is cellular. Given that cells are handlebodies of genus zero, a natural step beyond cellularity is to consider compacta $K \subseteq \mathbb{R}^n$ that have neighbourhood bases comprised of handlebodies of genus one; that is, solid tori. This is the starting point of this paper: we will concentrate on the three dimensional case and call compacta with the property just described \emph{toroidal}. Toroidal sets appear naturally among attractors for homeomorphisms of $\mathbb{R}^3$, with the dyadic solenoid being a well known example. Unlike the cells that define a cellular set, solid tori can be knotted and so it would seem that toroidal sets can also be ``knotted'' in some way. We explore this idea by defining the genus of a toroidal set as a natural extension of the classical genus of a knot. As we shall see, toroidal sets may have infinite genus. Our interest in this stems from its applications to the realizability problem stated above: we shall prove that toroidal attractors must have finite genus and, using this, we will be able to construct (uncountably) many different examples of toroidal sets that cannot be realized as attractors because they have infinite genus; namely some wild knots that are expressed as an infinite connected sum of non-trivial knots and also knotted solenoids. Knotted solenoids were studied by Conner, Meilstrup and Repo\v{v}s in \cite{Conner1}.

The paper is organized as follows. The first two sections are purely topological in nature, with no dynamics present yet. In Section \ref{sec:def} we define toroidal sets, give some examples and explore their most basic properties. In Section \ref{sec:genus} we define the genus of a toroidal set and compute it for some examples, illustrating that a toroidal set may have infinite genus. We also show how the classical Alexander polynomial of knot theory can be naturally generalized for toroidal sets having finite genus. Dynamics enter the picture in Section \ref{sec:dynamics}, where we show that a toroidal attractor must have finite genus and, drawing on the previous sections, construct families of toroidal sets that cannot be realized as local attractors for homeomorphisms of $\mathbb{R}^3$. In \cite{mio6} a topological invariant $r(K)$ was introduced with the same purpose as the genus here; namely showing that certain sets $K \subseteq \mathbb{R}^3$ cannot be realized as attractors because of the way they sit in ambient space. The final section of this paper discusses toroidal sets from the perspective of the topological invariant $r$.

Some basic knowledge about the fundamental group and singular and \v{C}ech homology and cohomology theories, including the Alexander duality theorem, is needed in order to understand the paper. Among the classical references for this topic we may cite the books by Massey \cite{Massey1}, Hatcher \cite{hatcher1}, Munkres \cite{munkres3} and Spanier \cite{spanier1}. We shall use the notation $H_*$ and $H^*$ to denote the singular homology and cohomology functors and denote by $\check{H}^*$ the \v Cech cohomology functor. Singular and \v{C}ech cohomology coincide over polyhedra, so we will sometimes use both interchangeably.

\section{Toroidal sets} \label{sec:def}
 
In this section we shall introduce the class of toroidal sets and study their basic features.

 \begin{definition}
Let $K \subseteq \mathbb{R}^3$ be a compact set. We say that $K$ is \emph{toroidal} if it is not cellular and has a neighbourhood basis comprised of solid tori. A solid torus is a set homeomorphic to $\mathbb{S}^1 \times \mathbb{D}^2$. 
\end{definition}

We remark that a set $K$ constructed as an intersection of a nested sequence of solid tori is not automatically toroidal, since one must also check that it is not cellular. However, in the present paper we will mostly deal with sets $K$ having $\check{H}^1(K) \neq 0$, and these are certainly not cellular.

\begin{example} Let $K \subseteq \mathbb{R}^3$ be a polyhedral knot. Any regular neighbourhood of $K$ (in the sense of piecewise linear topology) is a solid torus. Thus, $K$ is a toroidal set (as just mentioned, the fact that $K$ is not cellular follows from $\check{H}^1(K) = \mathbb{Z})$.
\end{example}

Recall that a knot is a set $K \subseteq \mathbb{R}^3$ homeomorphic to the circumference $\mathbb{S}^1$. Our first example of a toroidal set is a well known construction in knot theory:

\begin{example} \label{ex:wild} Consider arranging infinitely many polyhedral knots $K_i$ as suggested in Figure \ref{fig:conn_sum1}. Each knot $K_i$ is placed in a cell $C_i$; the size of these cells decreases to zero as they approach the point $p$. We claim that the resulting knot $K$ is toroidal. To prove this let $T_i$ be a very thin tube that follows $K_1$ up to $K_i$ very closely and then ``swallows'' the knots $K_{i+1},K_{i+2},\ldots$ as suggested by the grey outline in Figure \ref{fig:conn_sum2}. Clearly these $T_i$, if constructed appropriately, form a neighbourhood basis of $K$.

\begin{figure}[h]
\begin{pspicture}(0,0)(14,4)
\rput[bl](0,0){\scalebox{0.7}{\includegraphics{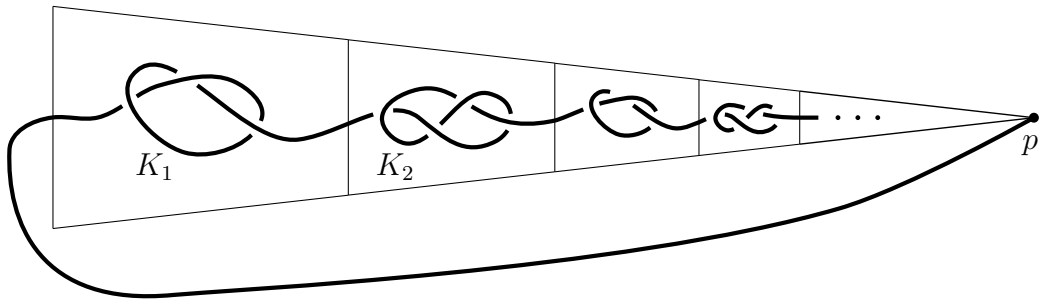}}}
\rput[bl](1.7,1.6){$K_1$} \rput[bl](4.9,1.6){$K_2$} \rput[t](13.6,2.2){$p$}
\end{pspicture}
\caption{Connected sum of infinitely many knots \label{fig:conn_sum1}}
\end{figure}

\begin{figure}[h]
\begin{pspicture}(0,0)(14,4)
\rput[bl](0,0){\scalebox{0.7}{\includegraphics{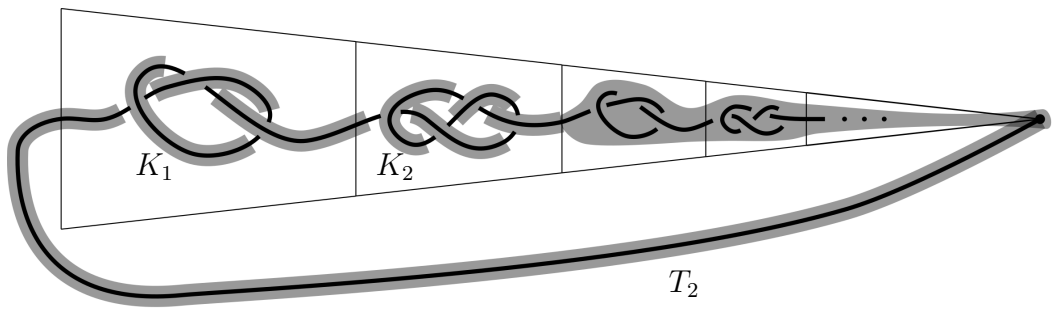}}}
\rput[bl](1.7,1.7){$K_1$} \rput[bl](4.9,1.7){$K_2$}
\rput[t](9,0.5){$T_2$}
\end{pspicture}
\caption{The torus $T_2$ \label{fig:conn_sum2}}
\end{figure}

\end{example}

The previous example is nothing but a connected sum of infinitely many polyhedral knots. A knot $K$ is \emph{tame} at a point $p\in K$ if there exists a neighbourhood $U$ of $p$ in $\mathbb{R}^3$ and a homeomorphism $h:\mathbb{R}^3\longrightarrow\mathbb{R}^3$ which carries $U\cap K$ onto a straight line. If $K$ is not tame at $p$ then it is said to be \emph{wild} at $p$. Notice that, by its very definition, the set of tame points is an open subset of $K$ and, hence, the set of wild points is compact. It is clear from the construction that the knot $K$ defined in Example \ref{ex:wild} above is locally polyhedral at each point except for, possibly, at the limit point $p$. On intuitive grounds one may suspect that $p$ is indeed a wild point of $K$, at least if the knots $K_i$ are truly nontrivial. We shall see later on that this is indeed the case.

\begin{example} \label{ex:solenoids} Consider a nested family $\{T_i\}_{i \geq 0}$ of solid tori in $\mathbb{R}^3$ whose thickness decreases to zero and such that $T_{i+1}$ winds $n_i$ times inside $T_i$, where $n_i \geq 2$. The set $K=\bigcap_{i \geq 0} T_i$ is a continuum known as a \emph{generalized solenoid} (sometimes more stringent conditions are placed on the diameters and placement of the $T_i$, but this very general definition will be enough for our purposes). If there exists a non-negative integer $n$ such that $n_i=n$ for every $i$ big enough the set $X$ is usually called an $n$--adic solenoid. Solenoids have nonvanishing $\check{H}^1$ with $\mathbb{Z}$ coefficients (see the computations below), so they are not cellular; thus, they are toroidal sets.
\end{example}
 
\begin{example} \label{ex:whitehead} Let $T_0$ be a standard solid torus in $\mathbb{R}^3$ and let $h:\mathbb{R}^3\longrightarrow\mathbb{R}^3$ be a homeomorphism such that $h(T_0)$ is contained in the interior of $T_0$ as indicated in Figure~\ref{fig:wh}. Consider the decreasing family of solid tori $\{T_i\}$, where $T_i=h^i(T_0)$, and let $K$ be the intersection of this family. The resulting continuum $K$ is usually called the \emph{Whitehead continuum}. In this case it is also true, but not easy to prove, that $K$ is not cellular \cite[p. 156]{hempel1}. Thus the Whitehead continuum is indeed a toroidal set.
\begin{figure}
\includegraphics[scale=1]{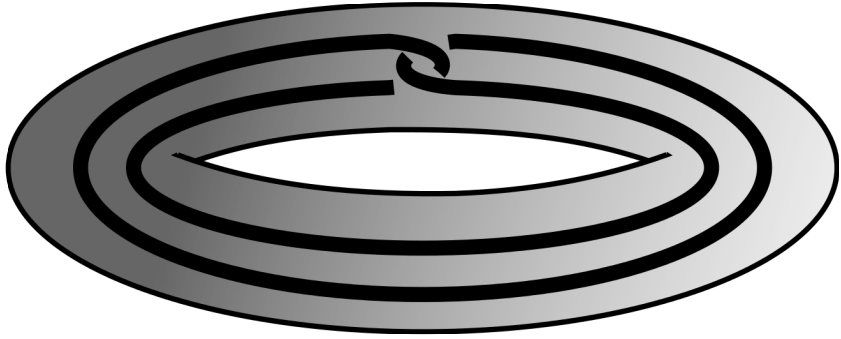}
\caption{The Whitehead continuum}
\label{fig:wh}
\end{figure}
\end{example}

Now we turn to analyze the (\v{C}ech) cohomology of toroidal sets. Suppose $K \subseteq \mathbb{R}^3$ is a toroidal set and let $\{T_i\}$ be a neighbourhood basis of $K$ comprised of nested solid tori. By nested we mean that $T_{i+1}$ is contained in the interior of $T_i$ for each $i$. Each inclusion $T_{i+1} \subseteq T_i$ induces a homomorphism $\alpha_i : H^1(T_i;\mathbb{Z}) \longrightarrow H^1(T_{i+1};\mathbb{Z})$. The groups $H^1(T_i;\mathbb{Z})$ are all isomorphic to $\mathbb{Z}$, so every $\alpha_i$ can be thought of as multiplication by an integer $w_i$ which is uniquely determined up to sign; in fact, suitably choosing the generators of $H^1(T_i;\mathbb{Z})$ we can assume that the $w_i$ are all nonnegative. We call $w_i$ the \emph{winding number} of $T_{i+1}$ inside $T_i$.

\begin{proposition} \label{prop:coh} A toroidal set is connected and has trivial \v{C}ech cohomology in degrees $q \geq 2$. As for its cohomology in degree one, the following alternative holds (the notation is the same as in the previous paragraph):
\begin{itemize}
	\item[(1)] Suppose infinitely many of the $w_i$ are zero. Then $\check{H}^1(K;\mathbb{Z}) = 0$.
	\item[(2)] Suppose that $w_i = 1$ from some $i_0$ on. Then $\check{H}^1(K;\mathbb{Z}) = \mathbb{Z}$.
	\item[(3)] Suppose that $w_i \geq 2$ for infinitely many $i$. Then $\check{H}^1(K;\mathbb{Z})$ is not finitely generated.
\end{itemize}
\end{proposition}
\begin{proof} Toroidal sets are connected because they are the intersection of a nested sequence of tori, which are compact and connected. The first \v{C}ech cohomology group of $K$ is isomorphic to the direct limit \begin{equation} \label{eq:coh} \check{H}^1(K;\mathbb{Z}) = \varinjlim\ \left\{ \xymatrix{\mathbb{Z} \ar[r]^{\cdot w_1} & \mathbb{Z} \ar[r]^{\cdot w_2} & \ldots \ar[r] & \mathbb{Z} \ar[r]^{\cdot w_k} & \mathbb{Z} \ar[r] & \ldots} \right\} \end{equation} With this description cases (1) and (2) should be clear. As for (3), suppose that $w_i \geq 2$ for infinitely many $i$ but $\check{H}^1(K;\mathbb{Z})$ were finitely generated. Then there would exist a finite family $z_1, \ldots, z_r$ contained in some term (say, the $i_0$th) of \eqref{eq:coh} whose images in the direct limit $\check{H}^1(K;\mathbb{Z})$ generate it; since each $\alpha_i$ is injective (because $w_i \neq 0$) this would imply that all the bonding maps for $i \geq i_0$ must be surjective and hence $w_i = 1$.

The same sequence as in \eqref{eq:coh}, now for $q \geq 2$, shows that $\check{H}^q(K) = 0$ in those degrees.
\end{proof}

When case (1) holds we shall say that $K$ is a \emph{homologically trivial} or, simply, \emph{trivial} toroidal set;  we call it \emph{nontrivial} otherwise.

\begin{example} (1) The Whitehead continuum $K$ of Example \ref{ex:whitehead} is a trivial toroidal set. Observe that the winding number of $T_1 = h(T_0)$ inside $T_0$ is zero, so the same holds for each $T_{i+1}$ inside the previous $T_i$ since all the pairs $(T_i,T_{i+1})$ are homeomorphic to each other (via $h$) by construction. In particular the bonding maps in \eqref{eq:coh} are all zero, so $K$ is indeed a trivial toroidal set.

(2) Any generalized solenoid $K$ has, in the notation of Example \ref{ex:solenoids}, winding numbers $w_i = n_i \geq 2$. Thus $\check{H}^1(K)$ is not finitely generated.

(3) The polyhedral knots and the wild knots of Example \ref{ex:wild} are all homeomorphic to $\mathbb{S}^1$; thus, they have $\check{H}^1(K) = \check{H}^1(\mathbb{S}^1) = \mathbb{Z}$ since \v{C}ech cohomology is a topological invariant.
\end{example}

So far no condition has been placed on the tori $T_i$ that comprise a neighbourhood basis of a toroidal set. We now make some considerations to show that the $T_i$ can always be chosen to satisfy certain convenient properties.

(i) A result of Moise \cite[Theorem 1, p. 253]{moise2} implies the following: given a compact $3$--manifold $N \subseteq \mathbb{R}^3$ and $\epsilon > 0$, there exists an embedding $h : N \longrightarrow \mathbb{R}^3$ that moves points less than $\epsilon$ and such that $h(N)$ is a polyhedron. A consequence of this result (letting $N = T_i$ and choosing $\epsilon$ so small that $h(T_i)$ is still a neighbourhood of $K$) is that a toroidal set admits a neighbourhood basis of solid \emph{polyhedral} tori.

(ii) Suppose that $K$ is a nontrivial toroidal set, so that $\check{H}^1(K) \neq 0$. A \emph{natural neighbourhood} of $K$ is a solid polyhedral torus $T$ such that the inclusion $K \subseteq T$ induces a nonzero map in $\check{H}^1$. We have the following result:

\begin{remark} \label{rem:natural} Let $K$ be a nontrivial toroidal set. Let $\{T_i\}$ be a neighbourhood basis of $K$ comprised of solid, polyhedral tori. Then, for $i$ big enough, each $T_i$ is a natural neighbourhood of $K$.
\end{remark}
\begin{proof} We only need to show that the inclusion $K \subseteq T_i$ induces a nonzero map in $\check{H}^1$ for big enough $i$. Consider the identity ${\rm id} : K \longrightarrow K$ as the inverse limit of the inclusions $j_i : K \subseteq T_i$. The direct limit of the induced maps $j_i^* : \check{H}^1(T_i) \longrightarrow \check{H}^1(K)$ is precisely the induced map ${\rm id}^* = {\rm id} : \check{H}^1(K) \longrightarrow \check{H}^1(K)$. If $j_i^*$ were trivial for a subsequence of the $\{T_i\}$, it would follow that ${\rm id} : \check{H}^1(K) \longrightarrow \check{H}^1(K)$ would also be trivial, showing that $\check{H}^1(K)$ itself is trivial in contradiction with the assumption that $K$ be nontrivial.
\end{proof}

A \emph{natural neighbourhood basis} of $K$ will mean a sequence $\{T_i\}$ of nested natural neighbourhoods of $K$. As a consequence of the above discussion, every nontrivial toroidal set has a natural neighbourhood basis. Natural neighbourhood bases have two properties that we will often use in the sequel without further explanation:

\begin{remark} \label{rem:natural1} Let $\{T_i\}$ be a natural neighbourhood basis of a nontrivial toroidal set $K$. Then:
\begin{itemize}
	\item[(i)] The winding number of each $T_{i+1}$ inside $T_i$ is nonzero.
	\item[(ii)] If $\check{H}^1(K;\mathbb{Z}) = \mathbb{Z}$, then for big enough $i$ the winding number of $T_{i+1}$ inside $T_i$ is precisely one.
\end{itemize}
\end{remark}
\begin{proof} (i) Since the inclusion $K \subseteq T_i$ factors through the inclusion $T_{i+1} \subseteq T_i$ and the former induces a nonzero map in $\check{H}^1$ because $T_i$ is a natural neighbourhood of $K$, the latter must also induce a nonzero map in $\check{H}^1$; that is, the winding number of $T_{i+1}$ inside $T_i$ must be nonzero.

(ii) This follows directly from our previous discussion about $\check{H}^1(K;\mathbb{Z})$.
\end{proof}

\section{The genus of a toroidal set} \label{sec:genus}

In this section we are going to define the genus of a toroidal set. We first recall very briefly some notions pertaining to classical knot theory. Further details can be found, for instance, in the books by Burde and Zieschang \cite{burdezieschang1}, Lickorish \cite{Lickorish1} or Rolfsen \cite{rolfsen1}.

Two knots $K$ and $K'$ are equivalent if there exists a homeomorphism $h:\mathbb{R}^3\longrightarrow\mathbb{R}^3$ such that $h(K)=K'$. In particular a knot $K$ is said to be \emph{unknotted} if it is equivalent to the standard $\mathbb{S}^1$ contained in the plane $z=0$ of $\mathbb{R}^3$.

Given a polyhedral knot $K\subseteq\mathbb{R}^3$ there exists an orientable, connected surface $S\subseteq\mathbb{R}^3$ whose boundary is precisely $K$. Such a surface is called a \emph{Seifert surface} of the knot. The \emph{genus} of $K$ is defined as the minimal genus of a Seifert surface of $K$. Clearly the genus is an invariant of the knot type; that is, if two knots are equivalent then they have the same genus. The only orientable, connected surface with genus zero is a disk; hence, a knot has genus zero if and only if it bounds a disk or, equivalently, if and only if it is the unknot.

Let $T$ be a solid torus and $h : \mathbb{S}^1 \times \mathbb{D}^2 \longrightarrow T \subseteq \mathbb{R}^3$ a specific homeomorphism between $T$ and $\mathbb{S}^1 \times \mathbb{D}^2$. Such a homeomorphism is called a \emph{framing} of $T$. As mentioned earlier in the previous section, without of loss generality we will always assume that solid tori $T \subseteq \mathbb{R}^3$ are polyhedral; similarly, it will also be convenient to think of $h$ as being piecewise linear with respect to the standard piecewise linear structure on $\mathbb{S}^1 \times \mathbb{D}^2$. The two simple closed curves $h(* \times \partial \mathbb{D}^2)$ and $h(\mathbb{S}^1 \times *)$, where $*$ denotes respectively a point in $\mathbb{S}^1$ or in $\partial \mathbb{D}^2$, are called a meridian and a longitude of $T$. Of course, they depend on the framing $h$. A solid torus $T$ that is given only as a point set admits many different framings; however, up to isotopy there are only two (depending on the orientation of the longitude) with the property that the associated longitude is nullhomologous in the complement of $T$. This is called the \emph{preferred framing} of $T$ (see \cite[Section 3.A, p. 30ff.]{burdezieschang1} or \cite[Section 2.E, p. 29ff.]{rolfsen1} for more details). The \emph{core} of $T$ is the simple closed curve $h(\mathbb{S}^1 \times 0)$, where $0$ is the center of the disk $\mathbb{D}^2$. Up to isotopy within $T$, and therefore also within $\mathbb{R}^3$, the core curve does not depend on $h$. Conversely, given a simple closed curve in $\mathbb{R}^3$, it is the core curve of any regular neighbourhood of itself, which is a solid torus uniquely determined up to isotopy. Thus it makes both heuristic and mathematical sense to consider $T$ and its core curve as equivalent as far as the study of the knottedness of $T$ is concerned, and from now on we shall freely apply definitions and techniques pertaining to knot theory to polyhedral solid tori. In particular, we shall say that a solid torus $T$ is unknotted if its core curve is unknotted. Similarly, we define the genus of $T$ as the genus of its core curve, and denote it by $g(T)$.

\begin{definition} Let $K \subseteq \mathbb{R}^3$ be a toroidal set. We shall say that $K$ is unknotted if it has a neighbourhood basis of unknotted solid tori.
\end{definition}



We define the genus of a toroidal set in the following rather natural way:

\begin{definition} Let $K \subseteq \mathbb{R}^3$ be a toroidal set. We define the \emph{genus} of $K$ as the minimum $g$ among $0,1,2,\ldots,\infty$ such that $K$ has arbitrarily small neighbourhoods that are polyhedral solid tori of genus $\leq g$.
\end{definition}

\begin{example} \label{ex:unknotted} Let $K \subseteq \mathbb{R}^3$ be a toroidal set. It follows directly from the definitions that $K$ is unknotted if, and only if, its genus is zero.
\end{example}

From its definition it is easy to bound the genus of a toroidal set from above, but not (in principle) to compute it exactly. The following theorem is very helpful in this regard:

\begin{theorem} \label{teo:compute_genus} Let $K \subseteq \mathbb{R}^3$ be a nontrivial toroidal set. Then its genus can be computed as \[g(K) = \lim_{i \rightarrow +\infty} g(T_i),\] where $\{T_i\}$ is any neighbourhood basis of $K$ comprised of nested, polyhedral solid tori.
\end{theorem}

To prove the theorem we need to recall some results concerning satellite knots (tailored to our context of solid tori). Let $T$ and $T'$ be two polyhedral solid tori such that $T'$ is contained in the interior of $T$. Let $V := \mathbb{S}^1 \times \mathbb{D}^2 \subseteq \mathbb{R}^3$ be the standard unknotted solid torus and let $e : T \longrightarrow V$ be the inverse of a preferred framing. The torus $T'$ gets sent by $e$ onto a solid torus $e(T') \subseteq {\rm int}\ V$. We call the pair $(V,e(T'))$ the \emph{pattern} of $(T,T')$ and define its genus $g(T,T')$ as the genus of $e(T')$. The genus of $(T,T')$ is well defined. For, consider two maps $e,e' : T \longrightarrow V$ that are inverses of preferred framings. Then the composition $e' e^{-1} : V \longrightarrow V$ is a homeomorphism that preserves the standard longitude of $V$ up to isotopy (by the definition of preferred framing) and therefore admits an extension to all of $\mathbb{R}^3$. This extension sends $e(T')$ onto $e'(T')$, showing that both knots are equivalent and in particular have the same genus. We shall say that the pattern of the pair $(T,T')$ is trivial if $g(T,T') = 0$. This means that $e(T')$ has genus zero; that is, $e(T')$ unknots in $\mathbb{R}^3$ (notice, however, that this does not necessarily imply that $e(T')$ can be unknotted within $V$).

Now suppose that (a) $T'$ lies in a nontrivial way inside $T$, this meaning that there does not exist a $3$--cell $B$ such that $T' \subseteq B \subseteq T$, and (b) $T$ is not unknotted. Then (the core curve of) $T'$ is called a satellite knot and (the core curve of) $T$ is its companion knot (see for instance \cite[Definition 2.8, p. 19]{burdezieschang1}). A well known result of Schubert establishes the inequality $g(T') \geq w \cdot g(T) + g(T,T')$, where $w$ is the winding number of $T'$ inside $T$. For a proof, see \cite[Proposition 2.10, p. 21]{burdezieschang1} (there is a typographical error in the statement of the proposition, with the correct formula being given at the end of its proof, in p. 22). 

The following result is a straightforward consequence of the above definitions and the formula of Schubert:

\begin{lemma} Let $T', T \subseteq \mathbb{R}^3$ be two solid tori such that $T' \subseteq {\rm int}\ T$. Suppose that the winding number $w$ of $T'$ inside $T$ is nonzero. Then the genera of $T$, $T'$ and the pair $(T,T')$ satisfy the inequality \begin{equation} \label{eq:schubert} g(T') \geq w \cdot g(T) + g(T,T').\end{equation}
\end{lemma}
\begin{proof} The assumption that $w$ is nonzero implies that $T'$ lies in a nontrivial way inside $T$ (otherwise the inclusion $T' \subseteq T$ would factor through a $3$--cell $B$ and the induced map in cohomology would factor through $\check{H}^1(B) = 0$, yielding the zero map). If $T$ is not the unknot, then (a) and (b) are satisfied, $T'$ is a satellite knot with companion $T$, and Equation \eqref{eq:schubert} is just the inequality of Schubert mentioned above.

Now suppose that $T$ is the unknot. By definition we cannot any longer properly speak of a satellite knot because condition (b) is violated. However, Equation \eqref{eq:schubert} remains valid since it reduces to $g(T') \geq g(T,T')$, which we can show directly to be true as follows. Let $e : T \longrightarrow V$ be the inverse of a preferred framing, so that $(V,e(T'))$ is the pattern of $(T,T')$ and therefore $g(T,T')$ is the genus of $e(T')$ by definition. Since both $T$ and $V$ are the unknot, $e$ has an extension to all of $\mathbb{R}^3$, say $\hat{e} : \mathbb{R}^3 \longrightarrow \mathbb{R}^3$ (again, we have implicitly used that $e$ sends the standard longitude of $T$ onto the standard longitude of $V$ because it is the inverse of a preferred framing). This $\hat{e}$ sends $T'$ onto $\hat{e}(T') = e(T')$, so both have the same genus. Thus $g(T') = g(T,T')$ as was to be shown.
\end{proof}

\begin{proposition} \label{prop:limit} Let $K$ be a nontrivial toroidal set and let $\{T_i\}$ be a natural neighbourhood basis for $K$. Denote by $g_i$ the genus of $T_i$. Then:
\begin{itemize}
	\item[(i)] The sequence $\{g_i\}$ is nondecreasing.
	\item[(ii)] The limit of the sequence $\{g_i\}$ (possibly $+\infty$) is independent of the particular neighbourhood basis $\{T_i\}$ chosen to compute it.
\end{itemize}
\end{proposition}
\begin{proof} (i) Since $\{T_i\}$ is a natural neighbourhood basis, the winding number $w_i$ of each $T_{i+1}$ inside $T_i$ satisfies $w_i \geq 1$. It then follows directly from Equation \eqref{eq:schubert} that $g_{i+1} \geq w_i \cdot g_i \geq g_i$.

(ii) Let $\{T'_j\}$ be another natural neighbourhood basis for $K$. It will be enough to show that for each $i_0$ there exists $j_0$ such that $g(T'_j) \geq g(T_{i_0})$ for every $j \geq j_0$: this implies that $\lim_{j \rightarrow +\infty} g(T'_j) \geq g(T_{i_0})$ and then letting $i_0 \rightarrow +\infty$ yields $\lim_{j \rightarrow +\infty} g(T'_j) \geq \lim_{i \rightarrow +\infty} g(T_i)$. The reverse inequality follows by interchanging the roles of $\{T_i\}$ and $\{T'_j\}$.

Fix, then, $i_0$. Choose $j_0$ so big that $T'_j \subseteq {\rm int}\ T_{i_0}$ for every $j \geq j_0$. Now fix $j \geq j_0$ and choose $k$ such that $T_k \subseteq {\rm int}\ T'_j$. The winding number of $T_k$ inside $T_{i_0}$ is the product of the winding numbers of $T_k$ inside $T'_j$ and of $T'_j$ inside $T_{i_0}$. Since $\{T_i\}$ is a natural neighbourhood basis the former is nonzero (by Remark \ref{rem:natural1}); thus, the latter are also nonzero. Then Equation \eqref{eq:schubert} implies that $g(T'_j) \geq g(T_{i_0})$, as was to be shown.
\end{proof}

Theorem \ref{teo:compute_genus} is a direct consequence of Remark \ref{rem:natural} and Proposition \ref{prop:limit}.(ii).

\begin{example} \label{ex:knots} Regard a polyhedral knot $K \subseteq \mathbb{R}^3$ as a toroidal set. To compute its genus as a toroidal set, let $T_i$ be progressively smaller regular neighbourhoods of $K$ and observe that all the $T_i$ have the same genus, which is precisely the genus of $K$ as a knot. These $\{T_i\}$ form a natural neighbourhood basis for $K$, so from Theorem \ref{teo:compute_genus} we see that the genus of $K$ as a toroidal set and as a knot coincide.
\end{example}

\begin{example} \label{ex:wild1} Consider again the infinite connected sum of polyhedral knots described in Example \ref{ex:wild}. The tori $\{T_i\}$ described there form a natural neighbourhoood basis of $K$. The core curve of $T_i$ is the connected sum $K_1 \# \ldots \# K_i$, and so its genus is $g(T_i) = \sum_1^i g(K_i)$ because genus is additive. It follows from Theorem \ref{teo:compute_genus} that the genus of $K$ is $\sum_1^{\infty} g(K_i)$. Notice that this is finite if and only if $g(K_i) = 0$ for big enough $i$; that is, if and only if the $K_i$ are unknotted from some $i_0$ onwards.
\end{example}

\begin{theorem} \label{teo:fin_gen} Let $K$ be a nontrivial toroidal set having finite genus. Then:
\begin{itemize}
	\item[(i)] There exist a natural neighbourhood $T$ of $K$ and an embedding $e : T \longrightarrow \mathbb{R}^3$ that unknots $K$; that is, such that $e(K) \subseteq \mathbb{R}^3$ is an unknotted toroidal set.
	\item[(ii)] If $\check{H}^1(K;\mathbb{Z}) \neq \mathbb{Z}$, then $K$ is unknotted.
\end{itemize}
\end{theorem}

Intuitively, part (i) of the theorem means that a toroidal set with finite genus becomes unknotted as soon as one unknots one of its natural neighbourhoods. The converse to this is false, as we illustrate in Example \ref{ex:knotted_sol} below. Also, notice that part (ii) implies that a generalized solenoid with finite genus must be unknotted.

\begin{proof} Let $\{T_i\}$ be a natural neighbourhood basis of $K$ and denote by $g_i$ the genus of $T_i$. Since the genus of $K$ is finite, it follows from Proposition \ref{prop:limit}.(i) that the sequence $\{g_i\}$ eventually becomes constant; that is, there exists $i_0$ such that $g_i = g_{i_0}$ for every $i \geq i_0$.

(i) Set $T := T_{i_0}$ and let $e$ be the inverse of a preferred framing for $T$. Let $i > i_0$. The winding number $w$ of $T_i$ inside $T_{i_0}$ is nonzero by Remark \ref{rem:natural1}. Hence we may apply Equation \eqref{eq:schubert} to the pair $(T_{i_0},T_i)$ and write $g_{i_0} = g_i \geq w \cdot g_{i_0} + g(T_{i_0},T_i)$. Since $w \geq 1$, it follows that $g_{i_0} \geq g_{i_0} + g(T_{i_0},T_i)$, whence $g(T_{i_0},T_i) = 0$. This means that the genus of $e(T_i)$ is zero. Thus $e(K)$ has the natural neighbourhood basis $\{e(T_i)\}_{i \geq i_0}$ all of whose members are unknotted; that is, $e(K)$ is unknotted.

(ii) We reason by contradiction. Suppose that there exists $i \geq i_0$ such that $g_i > 0$. Then for any $j \geq i$ we have, using Equation \eqref{eq:schubert} repeatedly, $g_i = g_{j+1} \geq w_j \cdot w_{j-1} \cdot \ldots \cdot w_{i+1} \cdot w_i \cdot g_i$ where as usual $w_j$ denotes the winding number of $T_{j+1}$ inside $T_j$. Since $g_i > 0$, it follows that $w_j \cdot \ldots \cdot w_{i+1} \cdot w_i \leq 1$ and, since none of these winding numbers is zero because each $T_i$ is a natural neighbourhood of $K$, they must all be equal to one. This implies that $\check{H}^1(K;\mathbb{Z}) = \mathbb{Z}$, contradicting the assumption. Thus $g_i = 0$ for every $i \geq i_0$, so all those $T_i$ are unknotted and so is $K$ by definition.
\end{proof}

Regarding unknotted toroidal sets, the characterizations given in the following proposition will be useful later on. Part (iii) is the analogue of the foundational result in knot theory which says that a knot $K$ is trivial if, and only if, the fundamental group of its complement $\mathbb{R}^3 - K$ is $\mathbb{Z}$.

\begin{proposition} \label{prop:unknotted} Let $K$ be a nontrivial toroidal set. Then the following are equivalent:
\begin{itemize}
	\item[(i)] $K$ is unknotted.
	\item[(ii)] Every natural neighbourhood of $K$ is unknotted.
	\item[(iii)] The fundamental group $\pi_1(\mathbb{R}^3 - K)$ is Abelian.
\end{itemize}
\end{proposition}
\begin{proof} (i) $\Rightarrow$ (ii) Let $T$ be a natural neighbourhood of $K$ and let $T'$ be an unknotted solid torus that is a neighbourhood of $K$ contained in the interior of $T$. The inclusion $K \subseteq T$ factors through the inclusion $T' \subseteq T$, so it follows that the latter induces a nonzero homomorphism in cohomology. In particular $T'$ winds at least once around $T$, so by Equation \eqref{eq:schubert} we have $g(T') \geq g(T)$. However $g(T') = 0$ because $T'$ is unknotted, so $g(T) = 0$ and $T$ is unknotted too.

(ii) $\Rightarrow$ (iii) Let $\{T_i\}$ be a natural neighbourhood basis of $K$ consisting of polyhedral tori (this exists by Remark \ref{rem:natural}). Since $\mathbb{R}^3 - K$ is the ascending union of the open sets $\mathbb{R}^3 - T_i$, it follows from well known properties of the fundamental group that $\pi_1(\mathbb{R}^3 - K)$ is the direct limit \begin{equation} \label{eq:dirlim} \pi_1(\mathbb{R}^3 - K) = \varinjlim\ \left\{ \xymatrix{\pi_1(\mathbb{R}^3 - T_1) \ar[r] & \pi_1(\mathbb{R}^3 - T_2) \ar[r] & \ldots } \right\} \end{equation} where the unlabeled arrows denote inclusion induced homomorphisms. By assumption each $T_i$ is unknotted, and so $\pi_1(\mathbb{R}^3 - T_i) = \mathbb{Z}$. Equation \eqref{eq:dirlim} then exhibits $\pi_1(\mathbb{R}^3 - K)$ as the direct limit of a direct sequence of Abelian groups; hence, $\pi_1(\mathbb{R}^3 - K)$ must be Abelian itself.

(iii) $\Rightarrow$ (i) As in the previous paragraph, let $\{T_i\}$ be a natural neighbourhood basis of $K$ and write $\pi_1(\mathbb{R}^3 - K)$ as the direct limit in Equation \eqref{eq:dirlim}. Suppose that $\pi_1(\mathbb{R}^3-K)$ is Abelian. Since $K$ is nontrivial, the winding number of each $T_{i+1}$ inside $T_i$ is nonzero (that is, each $T_{i+1}$ is a satellite of $T_i$), which by \cite[Theorem 9, p. 113]{rolfsen1} entails that the inclusion induced homomorphism $\pi_1(\mathbb{R}^3 - T_i) \longrightarrow \pi_1(\mathbb{R}^3 - T_{i+1})$ is injective for every $i$. It follows from Equation \eqref{eq:dirlim} that the same is true of the inclusion induced homomorphism $\pi_1(\mathbb{R}^3-T_i) \longrightarrow \pi_1(\mathbb{R}^3 - K)$. Since the latter group is Abelian by assumption, $\pi_1(\mathbb{R}^3 - T_i)$ is also Abelian and thus $T_i$ is unknotted by the classical result in knot theory mentioned above. Thus $K$, having a neighbourhood basis of unknotted solid tori, is unknotted too.
\end{proof}

The techniques used in the proof have some overlap with those in \cite[Lemma~2.2, p. 1550069-5]{Conner1}, where the authors undertake a detailed study of the fundamental group of the complement of generalized solenoids.

\begin{example} \label{ex:knotted_sol} We saw earlier (Theorem \ref{teo:fin_gen}.(i)) that a nontrivial toroidal set with finite genus can be reembedded (together with one of its neighbourhoods) in an unknotted fashion in $\mathbb{R}^3$. This example shows that the converse is not true.

Consider the standard dyadic solenoid $K'$ in $\mathbb{R}^3$ constructed as an intersection of a nested sequence of unknotted solid tori $\{V_i\}_{i \geq 0}$ each of which winds twice inside the previous one. Let $h : V_0 \longrightarrow \mathbb{R}^3$ be an embedding of $V_0$ into $\mathbb{R}^3$ knotted in some nontrivial way. We claim that $K := h(K')$ has infinite genus, although obviously by construction it satisfies the property stated in Theorem \ref{teo:fin_gen}.(i) (setting $e := h^{-1}$). We reason by contradiction. If $K$ had finite genus, since $\check{H}^1(K;\mathbb{Z}) = \check{H}^1(K';\mathbb{Z})$ is the dyadic rationals (in particular, it is neither $0$ nor $\mathbb{Z}$), it would be unknotted by Theorem \ref{teo:fin_gen}.(ii). But then, by Proposition \ref{prop:unknotted}, every natural neighbourhood of $K$ should be unknotted. Clearly $h(V_0)$ is one such neighbourhood, but it is not unknotted by construction. Thus the genus of $K$ must be infinite.
\end{example}

Let $K$ be a nontrivial toroidal set having finite genus. Since $K$ is nontrivial, $\check{H}^1(K;\mathbb{Z})$ is either not finitely generated or isomorphic to $\mathbb{Z}$. According to Theorem \ref{teo:fin_gen}.(ii), in the first case $K$ is unknotted. Now we devote a few lines to show that in the second case the classical Alexander polynomial from knot theory can be defined in a natural way for $K$. The following results will not be used elsewhere in the paper.

We denote the Alexander polynomial of a knot $K$ by $\Delta_K(t)$. Recall that the Alexander polynomial is an element in the polynomial ring $\mathbb{Z}[t,t^{-1}]$ and is defined only up to multiplication by a unit; that is, by a factor of the form $\pm t^n$. The Alexander polynomial of a polyhedral solid torus $T$ is defined as the Alexander polynomial of its core curve.

For the following proposition we will make use of the relation between the Alexander polynomials of a satellite, its companion, and its pattern \begin{equation} \label{eq:alex} \Delta_{\text{satellite}}(t) \stackrel{\cdot}{=} \Delta_{\text{pattern}}(t) \cdot \Delta_{\text{companion}}(t^w),\end{equation} where $w$ is the winding number of the satellite inside a tubular neighbourhood of its companion and $\stackrel{\cdot}{=}$ means ``equal up to multiplication by a unit''; that is, a factor $\pm t^n$.

\begin{theorem} \label{teo:alex} Let $K$ be a nontrivial toroidal set with finite genus and $\check{H}^1(K;\mathbb{Z}) = \mathbb{Z}$. There exists a polynomial $\Delta(t) \in \mathbb{Z}[t,t^{-1}]$ such that for every neighbourhood basis $\{T_i\}$ of $K$ consisting of solid polyhedral tori, the Alexander polynomials of the $T_i$ are all equal to $\Delta(t)$ (up to multiplication by units) for big enough $i$.
\end{theorem}

We shall make use of the following auxiliary proposition:

\begin{proposition} \label{prop:alex} Let $K$ be a nontrivial toroidal set with finite genus and $\check{H}^1(K;\mathbb{Z}) = \mathbb{Z}$. Let $\{T_i\}$ be a neighbourhood basis of $K$ consisting of polyhedral solid tori. Then for big enough $i$ all the $T_i$ have the same Alexander polynomial (up to multiplication by units).
\end{proposition}
\begin{proof} By Remark \ref{rem:natural} we may assume without loss of generality that the $T_i$ are natural neighbourhoods of $K$. Suppose first that they are nested. 

Since the genera of $T_i$ form an increasing sequence by Proposition \ref{prop:limit} and the genus of $K$ is finite by assumption, there exists $i_0$ such that $g(T_i)$ is constant for $i \geq i_0$. In particular we have from Equation \eqref{eq:schubert} that $g(T_i) = g(T_{i+1}) \geq w_i \cdot g(T_i) + g(T_i,T_{i+1})$ where $w_i$ is the winding number of $T_{i+1}$ inside $T_i$. Since $w_i \geq 1$ because the $\{T_i\}$ are natural neighbourhoods of $K$, it is apparent from this inequality that $g(T_i,T_{i+1}) = 0$. Thus, the pattern of $(T_i,T_{i+1})$ is trivial. In particular, the Alexander polynomial $\Delta_{(T_i,T_{i+1})} \stackrel{\cdot}{=} 1$. The condition that $\check{H}^1(K;\mathbb{Z}) = \mathbb{Z}$ implies that, for big enough $i$, the winding number of each $T_{i+1}$ inside the previous $T_i$ is precisely one. Then Equation \eqref{eq:alex} yields $\Delta_{T_i}(t) \stackrel{\cdot}{=} \Delta_{T_{i+1}}(t)$. Therefore, all the $T_i$ have the same Alexander polynomial $\Delta(t)$ (for big enough $i$).

Consider now the general case, where the $T_i$ are not necessarily nested. We reason by contradiction, so suppose that the sequence $\Delta_{T_i}(t)$ were not eventually constant. We construct a subsequence $\{T_{i_j}\}$ of the $\{T_i\}$ as follows. Start setting $T_{i_1} := T_1$ and find $i_2$ big enough so that (i) $T_{i_2} \subseteq {\rm int}\ T_{i_1}$ and (ii) $T_{i_1}$ and $T_{i_2}$ have a different Alexander polynomial. Condition (i) can be met because the $\{T_i\}$ are a neighbourhood basis of $K$; condition (ii) uses the assumption that the sequence $\Delta_{T_i}(t)$ is not eventually constant. Starting now with $T_{i_2}$ repeat the same process, and so on. This yields a subsequence $\{T_{i_j}\}$ of the $\{T_i\}$ such that (i) the $\{T_{i_j}\}$ are nested and (ii) every two consecutive $T_{i_j}$ and $T_{i_{j+1}}$ have a different Alexander polynomial. This, however, contradicts the previous paragraph.
\end{proof}

\begin{proof}[Proof of Theorem \ref{teo:alex}] Let $\{T_i\}$ and $\{T'_j\}$ be two neighbourhood bases of $K$ consisting of polyhedral tori. Combine the two to obtain a third neighbourhood basis for $K$, say $\{T''_k\}$, whose elements are defined as $T''_{2k-1} := T_k$ and $T''_{2k} := T'_k$. By Proposition \ref{prop:alex} applied to $\{T''_k\}$, for big enough $k$ all the $T''_k$ have the same Alexander polynomial. This readily implies that for big enough $i$ and $j$ all the $T_i$ and $T'_j$ have the same Alexander polynomial. This concludes the proof.
\end{proof}

The condition that the genus of $K$ be finite has only entered our arguments to guarantee that, if $\{T_i\}$ is any natural neighbourhood basis for $K$, then the pattern of $(T_i,T_{i+1})$ is trivial for big enough $i$. As it turns out, for Theorem \ref{teo:alex} to be true it is enough that this property holds for any one natural neighbourhood basis of $K$. Let us say that $K$ is \emph{weakly knotted} in that case. Since there exist weakly knotted toroidal sets having infinite genus, the Alexander polynomial can actually be defined for a wider class of toroidal sets than those of finite genus considered above.

\section{Applications (1)} \label{sec:dynamics}

Now we finally turn to dynamics and show that certain sets can be embedded in $\mathbb{R}^3$ in a knotted manner that makes it impossible to realize them as attractors. More specifically, we will show this to be the case for solenoids and the circumference $\mathbb{S}^1$; in fact, there are uncountably many nonequivalent embeddings with this property.

Let us recall some standard definitions first (see for instance \cite{hale1}). Let $f:\mathbb{R}^3\longrightarrow\mathbb{R}^3$ be a homeomorphism. A set $K\subseteq\mathbb{R}^3$ is called \emph{invariant} if $f(K)=K$. A compact set $P\subseteq\mathbb{R}^3$ is \emph{attracted} by $K$ if for each neighbourhood $V$ of $K$ there exists $n_0\in\mathbb{N}$ such that $f^n(P)\subseteq V$ whenever $n\geq n_0$. A compact invariant set $K$ is an attractor (a more precise but cumbersome terminology would be asymptotically stable attractor, or Lyapunov stable attractor) if it has a neighbourhood $U$ such that every compact set $P \subseteq U$ is attracted by $U$. \label{def:attract} The biggest $U$ for which this is satisfied is called the \emph{basin of attraction} of $K$ and is always an open invariant set. Some authors consider only global attractors or minimal attractors, but both restrictions are unnecesary in our context.

Among the examples of toroidal sets introduced in Section \ref{sec:def}, tame knots, $n$--adic solenoids and the Whitehead continuum can easily be seen to be attractors for suitable homeomorphisms of $\mathbb{R}^3$. The case of tame knots is intuitively plausible and can be proved using some elementary piecewise linear topology (see \cite[Corollary 4, p. 327]{gunthersegal1} or \cite[Proposition 12, p. 6169]{mio5}). As for $n$--adic solenoids and the Whitehead continuum, their realizability as attractors follows almost immediately from their definition. For instance, consider the case of the Whitehead continuum. In the notation introduced in Example \ref{ex:whitehead}, clearly $K$ is invariant for the homeomorphism $h$. Moreover, it attracts every compact subset $P$ of $U := {\rm int}\ T_0$: since $K$ is the decreasing intersection of the compact sets $h^i(T_0)$, for any neighbourhood $V$ of $K$ there exists $i_0$ such that $h^{i_0}(T_0) \subseteq V$, and then $h^i(P) \subseteq V$ for every $i \geq i_0$ and every $P \subseteq {\rm int}\ T_0 = U$. Thus $K$ is a stable attractor for $h$ and its basin of attraction contains ${\rm int}\ T_0$ (and in fact extends beyond it). The argument for $n$--adic solenoids is very similar. Generalized solenoids (Example \ref{ex:solenoids}) pose a more interesting problem: since the winding number $n_i$ is allowed to vary, it is not at all clear that such solenoids can be realized as attractors. In fact, not all of them can: G\"unther \cite{gunther1} proved that if the sequence $\{n_i\}$ consists of pairwise relatively prime integers, then the resulting generalized solenoid $K$ cannot be an attractor for a homeomorphism of $\mathbb{R}^3$. This result was obtained by a ingenious analysis of the cohomology group $\check{H}^1(K;\mathbb{Z})$. In this section we will shall see that knotted $n$--adic solenoids cannot be realized as attractors either, but the obstruction will now lie in their knottedness and not in their cohomology, which is the ``correct'' one for an attractor.

The main theoretical tool for this section is the following theorem. The reason why it is not entirely trivial is because some care has to be exercised concerning the polyhedral nature of the solid tori we are using throughout the paper.

\begin{theorem} \label{teo:finite} Let $K$ be a toroidal set that is an attractor for a homeomorphism $f$ of $\mathbb{R}^3$. Then the genus of $K$ is finite.
\end{theorem}
\begin{proof} Since $K$ is toroidal, there exists a polyhedral solid torus $T_0$ contained in the basin of attraction of $K$. Let $U$ be any compact neighbourhood of $K$ and choose $N$ big enough so that $f^N(T_0) \subseteq V$. Consider the compact set $f^N(\partial T_0)$, where $\partial T_0$ denotes the boundary of $T_0$. It is disjoint from both $\mathbb{R}^3 - V$ and $K$, so there exists $\epsilon > 0$ such that the $\epsilon$--neighbourhood of $f^N(T_0)$ is also disjoint from both sets. Let $g : \mathbb{R}^3 \longrightarrow \mathbb{R}^3$ be a piecewise linear homeomorphism that $\epsilon$--approximates $f^N$; that is, the distance between $f^N(p)$ and $g(p)$ is less than $\epsilon$ for every $p \in \mathbb{R}^3$ (for the existence of this piecewise linear approximation see for instance \cite[Theorem 1, p. 253]{moise2}). Let $T_1 := g(T_0)$. This is a polyhedral solid torus by construction. The choice of $g$ and $\epsilon$ guarantees that the boundary of $T_1$, which is $g(\partial T_0)$, is disjoint from both $\mathbb{R}^3 - V$ and $K$. Thus either $T_1$ contains $\mathbb{R}^3 - V$ and is disjoint from $K$ or viceversa. The first case is impossible, since the closure of $\mathbb{R}^3 - V$ is noncompact (because $V$ is compact) but $T_1$ is compact. Thus the $T_1$ contains $K$ and is disjoint from $\mathbb{R}^3 - V$; therefore, $T_1$ is a neighbourhood of $K$ contained in $V$. Moreover, by construction $T_1$ is ambient homeomorphic to $T_0$ and so both have the same genus. Repeating this construction while letting $V$ run in a neihbourhood basis of $K$ we see that $K$ has a neighbourhood basis of polyhedral solid tori, all ambient homeomorphic to $T_0$, and therefore its genus is bounded above by the genus of $T_0$ and is consequently finite. \end{proof}

The following corollary is a direct consequence of the above theorem and the preliminary work developed in the previous section:

\begin{corollary}\label{coro:attr} Let $K$ be a nontrivial toroidal set that is an attractor for a homeomorphism of $\mathbb{R}^3$. Then:
\begin{itemize}
	\item[(i)] If $\check{H}^1(K;\mathbb{Z}) \neq \mathbb{Z}$, then $K$ is unknotted. In particular, $\pi_1(\mathbb{R}^3 - K)$ is Abelian and every natural neighbourhood of $K$ is unknotted. 
	\item[(ii)] If $\check{H}^1(K;\mathbb{Z}) = \mathbb{Z}$, then $K$ has a well defined Alexander polynomial.
\end{itemize}
\end{corollary}

With this we can already give our first example of a solenoid that, when suitably reembedded in $\mathbb{R}^3$, cannot be realized as an attractor:

\begin{example} \label{ex:sol_not_attrac} The knotted dyadic solenoid of Example \ref{ex:knotted_sol} had, by construction, a knotted natural neighbourhood. Thus, it cannot be realized as an attractor for a homeomorphism of $\mathbb{R}^3$.
\end{example} 

The phenomenon illustrated in the previous example is actually very common, as the following corollary shows:

\begin{corollary} \label{cor:no_solenoid} Any solenoid can be embedded in $\mathbb{R}^3$ in uncountably many inequivalent ways such that none of them can be realized as an attractor for a homeomorphism of $\mathbb{R}^3$.
\end{corollary}

By \emph{nonequivalent} we mean the following. Let $e, e' : S \longrightarrow \mathbb{R}^3$ be two embeddings of the same compact topological space $S$ into $\mathbb{R}^3$. We shall say that $e$ and $e'$ are equivalent if there exists a homeomorphism $h$ of $\mathbb{R}^3$ that sends $e(S)$ onto $e'(S)$ (this is weaker than requiring that $e' = he$, which is another common notion of equivalence). Intuitively, if $e$ and $e'$ are equivalent then $e(S)$ and $e'(S)$ lie in $\mathbb{R}^3$ in the same way. Clearly, if two embeddings of $S$ are equivalent then $\mathbb{R}^3 - e(S)$ and $\mathbb{R}^3 - e'(S)$ are homeomorphic to each other.

\begin{proof} A result of Conner, Meilstrup and Repov\v{s} \cite[Theorem~5.4, p. 1550069-16]{Conner1} states that for any given solenoid $\Sigma$ (best thought as an abstract topological space) there exist uncountably many embeddings $e_{\alpha} : \Sigma \longrightarrow \mathbb{R}^3$, where $\alpha$ ranges in some uncountable set of indices $A$, such that: (1) the embedded solenoid $e_{\alpha}(\Sigma)$ is a toroidal set; (2) the fundamental group of $\mathbb{R}^3 - e_{\alpha}(\Sigma)$ is non Abelian and (3) if $\alpha \neq \beta$, then $\mathbb{R}^3 - e_{\alpha}(\Sigma)$ and $\mathbb{R}^3 - e_{\beta}(\Sigma)$ are non homeomorphic. This latter condition implies that the embeddings $e_{\alpha}(\Sigma)$ and $e_{\beta}(\Sigma)$ are non equivalent, and conditions (1) and (2) imply by Corollary \ref{coro:attr}.(i) that none of the embedded solenoids $e_{\alpha}(\Sigma)$ can be realized as an attractor for a homeomorphism of $\mathbb{R}^3$. This proves the corollary.
\end{proof}

This corollary is closely related to results of Jiang, Ni and Wang \cite{Jiang1}, who show that a closed orientable manifold admits a Smale knotted solenoid (this being a knotted solenoid with a very specific dynamics) as an attractor if and only if it has a lens space $L(p,q)$ with $p\neq 0,1$ as a prime factor. Corollary \ref{cor:no_solenoid} is more general in that we do not assume any prescribed dynamics on the solenoid but less general in that we only work in $\mathbb{R}^3$.

The previous examples involving solenoids may be relatively difficult to visualize. Now we move on to consider embeddings of $\mathbb{S}^1$ that cannot be realized as attractors either.	

\begin{example} \label{ex:infinite_sum} A knot constructed as an infinite connected sum of nontrivial knots as in Example \ref{ex:wild} has infinite genus, so by Theorem \ref{teo:finite} it cannot be realized as an attractor for a homeomorphism of $\mathbb{R}^3$.
\end{example}

Once again, this particular example illustrates a widespread phenomenon:

\begin{corollary} \label{cor:uncountably} There are uncountably many inequivalent ways of embedding $\mathbb{S}^1$ in $\mathbb{R}^3$ in such a way that it cannot be realized as an attractor for a homeomorphism. All these embeddings are tame except for a single point.
\end{corollary}

To prove the corollary we need the following auxiliary lemma. The technique used in the proof is standard and the lemma itself may be well known as there are similar results in the literature (see for instance \cite[Chapter 4]{yunlinhe} or \cite{foxharrold1}). We begin with a remark:

\begin{remark} \label{rem:wildpoint} Consider the following property (*) that may or may not hold for a given point $x$ of a knot $A$ in $\mathbb{R}^3$: if $\{V_n\}$ is a nested, closed neighbourhood basis of $x$ then there exists an index $N$ such that the inclusion induced homomorphism $\pi_1(V_N - A) \longrightarrow \pi_1(V_1 - A)$ has an Abelian image (see \cite[p. 988]{foxartin1} or \cite[p. 137]{moise2}). Returning to the infinite connected sum of Example \ref{ex:wild}, it is straightforward to see that this property holds for every point $x \in K - p$, and it can be shown by appropriately adapting the arguments in \cite[p. 988]{foxartin1} and \cite[p. 137 and 138]{moise2} that if (*) holds at $p$ then $K_i$ must be trivial from some $i$ onwards.
\end{remark}

\begin{lemma} Let $\{K_i\}$ and $\{K'_j\}$ be sequences of nontrivial, prime, tame knots and let $K$ and $K'$ be the toroidal sets constructed in Example \ref{ex:wild} as the infinite connected sums $\#_i K_i$ and $\#_j K'_j$. If $K$ and $K'$ are equivalent (that is, ambient homeomorphic) then each $K_i$ must appear in the sequence $\{K'_j\}$ (and viceversa).
\end{lemma}
\begin{proof} Suppose that $K$ and $K'$ are equivalent, so that there exists a homeomorphism $h$ of $\mathbb{R}^3$ that sends $K$ onto $K'$. Let $p$ and $p'$ be the limit points of the sequence of defining knots for $K$ and $K'$. Recall the property (*) mentioned in Remark \ref{rem:wildpoint} above and notice that by definition (*) is invariant under ambient homeomorphisms; thus, it holds at $x \in K$ if and only it holds at $h(x) \in K'$. Since (*) holds at every point of $K$ and $K'$ except at $p$ and $p'$, it follows that $h(p) = p'$.

Consider any one $K_{i_0}$. Consider also (in the notation of Example \ref{ex:wild}) the cell $\bigcup_{i \geq i_0+1} C_i \cup p$ and enlarge it a little bit to obtain a polyhedral closed ball $B$ which is a neighbourhood of $p$ and such that $K$ meets $\partial B$ transversely in exactly two points. Let $\alpha$ be a polyhedral arc in $\partial B$ that joins the two points in $K \cap \partial B$ and let $L :=  (K-\mathring{B})\cup \alpha$, which is a tame knot that is the connected sum $K_1 \# \ldots \# K_{i_0}$.

Now consider $h(B)$ and observe that it contains $h(p) = p'$. Thus, using again the notation of Example \ref{ex:wild} but this time with primed symbols because it concerns the set $K'$, there exists $j_0$ big enough so that the cell $\bigcup_{j \geq j_0+1} C'_j \cup p'$ is contained in the interior of $h(B)$. Enlarging this cell a little bit and splitting $K'$ along its boundary as in the previous paragraph, we obtain a tame knot $L'$ that is the connected sum $K'_1 \# \ldots \# K'_{j_0}$. It is also clear by construction that $h(L)$ is a connected summand of $L'$; thus, $K'_1 \# \ldots \# K'_{j_0} = h(K_1) \# \ldots \# h(K_{i_0}) \# (\text{some other connected summand})$. Since all the $K_i$ and $K'_j$ are prime, it follows that $h(K_{i_0})$ must coincide with some $K'_j$ with $1 \leq j \leq j_0$. This completes the proof.
\end{proof}

\begin{proof}[Proof of Corollary \ref{cor:uncountably}] For any pair $p > 1$ and $q > 1$ of coprime integers let $T_{p,q}$ denote the torus knot represented by a simple closed curve on the surface of a standard, unknotted torus in $\mathbb{R}^3$, winding $p$ times in the longitudinal direction and $q$ times in the meridional one. These knots are prime, nontrivial, and $T_{p,q}$ is equivalent to $T_{p',q'}$ if and only if $\{p,q\} = \{p',q'\}$.

For each $i = 1,2,\ldots$ let $K_i$ be the torus knot $T_{i+1,i+2}$. It follows from the above that the $K_i$ are nontrivial, prime and pairwise inequivalent. Let $S$ be the (uncountable) set of infinite sequences of zeroes and ones that contain infinitely many ones. For each sequence $s \in S$ consider the infinite connected sum of those $K_i$ such that the $i$th entry in $s$ equals one. The previous lemma entails that two different sequences $s, s' \in S$ give rise to inequivalent connected sums. Each of these is a toroidal set of infinite genus that cannot be realized as an attractor.
\end{proof}

Finally, another family of (perhaps not so natural) examples of nonattracting toroidal sets $K$ can be constructed in the following way. Suppose that we define $K$ as the intersection of a nested sequence of solid tori $\{T_i\}$ such that: (i) each torus winds at least once in the previous one and (ii) the pattern of $(T_i,T_{i+1})$ is nontrivial for infinitely many $i$. We may call such a $K$ a \emph{strongly knotted} toroidal set. Condition (i) guarantees that $K$ is nontrivial and the $\{T_i\}$ forms a natural neighbourhood basis for $K$. From Equation \eqref{eq:schubert} and Theorem \ref{teo:compute_genus} it then follows that the genus of $K$ is infinite; hence, by Theorem \ref{teo:finite} a strongly knotted toroidal set cannot be realized as an attractor for a homeomorphism.

Our goal is now to prove the following result:

\begin{theorem} \label{teo:separate} Let $K \subseteq \mathbb{S}^3$ be an attractor for a homeomorphism and let $K'$ be its dual repeller. Suppose that $\check{H}^1(K;\mathbb{Z}) \neq \mathbb{Z}$. Suppose also that there exists a $2$--torus $S \subseteq \mathbb{S}^3$ that separates $K$ and $K'$. Then both $K$ and $K'$ are unknotted toroidal sets.
\end{theorem}

Here a $2$--torus just means a homeomorphic copy of the surface of a solid torus; that is, a surface homeomorphic to $\mathbb{S}^1 \times \mathbb{S}^1$. The theorem does not require that $S$ be polyhedral. By Alexander duality $S$ separates $\mathbb{S}^3$ into two connected components whose topological frontiers are precisely $S$. When $S$ is a polyhedral surface then the closures of both of them are compact $3$--manifolds whose common boundary is precisely $S$.



\begin{proof}[Proof of Theorem \ref{teo:separate}] Let $U$ and $U'$ denote the two connected components of $\mathbb{S}^3 -S$ and assume that $K \subseteq U$ and therefore $K' \subseteq U'$. In particular $\overline{U} \cap K' = \emptyset$ and, since $K'$ is the complementary of $\mathcal{A}(K)$ in $\mathbb{S}^3$, it follows that $\overline{U} \subseteq \mathcal{A}(K)$. Similarly $\overline{U'} \subseteq \mathcal{R}(K')$.
\smallskip

{\it Claim 1.} Both $K$ and $K'$ are connected.

{\it Proof of claim.} $\overline{U}$ is a connected, compact neighbourhood of $K$ contained in its region of attraction. Choosing $N$ so big that $f^N(\overline{U}) \subseteq U$ we have $K = \bigcap_{k \geq 0} f^{Nk}(\overline{U})$. These $f^{Nk}(\overline{U})$, as $k$ runs in the natural numbers, form a nested sequence of compact, connected neighbourhoods of $K$. Since $K$ is the intersection of a nested family of continua, it is a continuum. The same proof works for $K'$, this time using $f^{-1}$ and $\overline{U'}$.
\smallskip

{\it Claim 2.} We may assume without loss of generality that $S$ is polyhedral.

{\it Proof of claim.} We need to make use a lemma whose proof is given after the theorem:

\begin{lemma} \label{lem:separate} Let $S \subseteq \mathbb{S}^3$ be a closed connected surface that separates two connected compact sets $K$ and $K'$. Let $h : S \longrightarrow \mathbb{S}^3$ be an embedding that moves every point less than some $\epsilon > 0$. For small enough $\epsilon$, the surface $S' := h(S)$ still separates $K$ and $K'$.
\end{lemma}

The claim is a consequence of the lemma and the following two results: given any $\epsilon > 0$, by the approximation theorem for surfaces of Bing \cite[Theorem 7, p. 478]{bing5} there exists an embedding $h : S \longrightarrow \mathbb{S}^3$ that moves points less than $\epsilon > 0$ and such that $S'$ is locally polyhedral at each point; (ii) a locally polyhedral closed subset of $\mathbb{R}^3$ is polyhedral \cite[Lemma 1, p. 146]{bing2}. Choosing an appropriately small $\epsilon > 0$ we may therefore replace $S$ with a polyhedral $S'$ that still separates $K$ and $K'$.
\smallskip

When $S$ is polyhedral, as we shall assume from now on according to Claim 2, both closures $\overline{U}$ and $\overline{U'}$ are polyhedral manifolds. Furthermore, as a consequence of a classical theorem of Alexander \cite[Theorem 1, p. 107]{rolfsen1}, at least one of $\overline{U}$ and $\overline{U'}$ is a solid torus. By Alexander duality $H_1(U \cup U') = \check{H}^{1}(S) = \mathbb{Z} \oplus \mathbb{Z}$; since $H_1(U \cup U') = H_1(U) \oplus H_1(U')$ and one of these summands is known to be $\mathbb{Z}$ (the one corresponding to the component that is known to be a solid torus), both must be $\mathbb{Z}$.
\smallskip

{\it Claim 3.} $\check{H}^1(K;\mathbb{Z})$ is not finitely generated.

{\it Proof of claim.} Suppose it were. Then by \cite[Theorem 2, p. 2827]{pacoyo1} the inclusion of $K$ in its region of attraction $\mathcal{A}(K)$ would induce isomorphisms in \v{C}ech cohomology with $\mathbb{Z}$ coefficients. The inclusion factors through $\overline{U}$, which has $\check{H}^1 = \mathbb{Z}$, so we would have \[\xymatrix{\check{H}^1(\mathcal{A}(K);\mathbb{Z}) \ar[r]^-{\alpha} & \check{H}^1(\overline{U};\mathbb{Z}) = \mathbb{Z} \ar[r]^-{\beta} & \check{H}(K;\mathbb{Z})}\] with $\beta \alpha$ an isomorphism. In particular $\alpha$ is injective, so that $\check{H}^1(\mathcal{A}(K);\mathbb{Z})$, being isomorphic to a subgroup of $\mathbb{Z}$, must be isomorphic to $\mathbb{Z}$ itself. Thus the same must be true of $\check{H}^1(K;\mathbb{Z})$, contradicting the assumption of the theorem.

\smallskip

{\it Claim 4.} $\check{H}^1(K';\mathbb{Z})$ is not finitely generated.

{\it Proof of claim.} If it were, by Alexander duality so would $H_1(\mathcal{A}(K);\mathbb{Z})$ and also, by the universal coefficient theorem, $H^1(\mathcal{A}(K);\mathbb{Z})$. But then, again by \cite[Theorem 2, p. 2827]{pacoyo1} the cohomology $\check{H}^1(K;\mathbb{Z})$ would be finitely generated too, contradicting the previous claim.
\smallskip

Now we can finish the proof of the theorem. By the result of Alexander mentioned earlier, at least one of $\overline{U}$ or $\overline{U'}$ is a solid torus and, accordingly, at least one of $K$ or $K'$ is a toroidal set (a neighbourhood basis of solid tori can be obtained by iterating $\overline{U}$ or $\overline{U'}$ with the dynamics, as in the proof of Theorem \ref{teo:finite}). In the first case we apply Corollary \ref{coro:attr}.(i) directly to $K$; in the second, we apply it to the toroidal attractor $K'$ for the reversed dynamics (notice that $K'$ satisfies the appropriate hypothesis concerning its $1$--cohomology by Claim 4 above). In either case we conclude that $K$ or $K'$ is an unknotted toroidal set. Suppose, for instance, that it is $\overline{U}$ that is a solid torus. We claim that $\overline{U}$ is actually a natural neighbourhood of $K$; that is, the inclusion $K \subseteq \overline{U}$ induces a nonzero homomorphism in $\check{H}^1(\, \cdot\, ;\mathbb{Z})$. As in the proof of Claim 1, consider again the nested family of compact, connected neighbourhoods $f^{Nk}(\overline{U})$ for an appropriately big $N$ and write $K = \bigcap_{k \geq 0} f^{Nk}(\overline{U})$. If the inclusion $f^{N(k+1)}(\overline{U}) \subseteq f^{Nk}(\overline{U})$ induced the zero homomorphism in $\check{H}^1(\, \cdot\, ;\mathbb{Z})$ for some $k$ then the same would be true for every $k$ and so $\check{H}^1(K;\mathbb{Z})$ would be zero, contradicting Claim 3. Thus each inclusion $f^{N(k+1)}(\overline{U}) \subseteq f^{Nk}(\overline{U})$ must induce a nonzero homomorphism in $\check{H}^1(\, \cdot\, ;\mathbb{Z})$, and therefore the same is true of the inclusion $K \subseteq \overline{U}$. Thus $\overline{U}$ is a natural neighbourhood of $K$, so by Corollary \ref{coro:attr}.(i) it must be an unknotted torus. The fact that it is unknotted guarantees that its complement $\overline{U'}$ is also an unknotted torus and, in turn, implies that $K'$ is also an unknotted toroidal set.
\end{proof}

\begin{proof}[Proof of Lemma \ref{lem:separate}] Since $S$ is a closed connected surface, it separates $\mathbb{S}^3$ into two connected components. The condition that $S$ separates $K$ and $K'$ amounts to saying that each of the components of $\mathbb{S}^3 - S$ intersects $K \cup K'$, or equivalently $H_0(\mathbb{S}^3 - S, K \cup K') = 0$. Letting $U$ and $U'$ be small connected open neighbourhoods of $K$ and $K'$ respectively, we also have $H_0(\mathbb{S}^3 - S , U \cup U') = 0$. By Alexander duality this is equivalent to $\check{H}^3(\mathbb{S}^3 - (U \cup U') , S) = 0$. In turn, through the exact sequence for the pair $(\mathbb{S}^3 - (U \cup U'),S)$, this is yet equivalent to the inclusion induced homomorphism $\check{H}^2(\mathbb{S}^3-(U \cup U')) \longrightarrow \check{H}^2(S)$ being surjective.

Consider the map $h$ and the inclusion $j : S \longrightarrow \mathbb{S}^3$ and construct the homotopy $H_t := (1-t) h + t j$, which as $t$ varies from $0$ to $1$ connects $h$ and $j$. The condition that $h$ moves points less than $\epsilon$ implies that the image of $H_t$ is contained in an $\epsilon$--neighbourhood of $S$. Therefore if $\epsilon$ is so small that such a neighbourhood is disjoint from $U$ and $U'$, the maps $h$ and $j$ are homotopic in $\mathbb{S}^3-(U \cup U')$. Consider $h$ and $j$ as embeddings of $S$ into $\mathbb{S}^3 - (U \cup U')$. By the previous paragraph the map $j^* : \check{H}^2(\mathbb{S}^3 - (U \cup U')) \longrightarrow \check{H}^2(S)$ is surjective, and therefore so is $h^*$. Let $j' : S' \longrightarrow \mathbb{S}^3 - (U \cup U')$ be the inclusion and denote by $h_0$ the map $h$ thought of as a homeomorphism between $S$ and $S' = h(S)$; that is, $h_0 : S \longrightarrow S'$. Clearly $h= j' h_0$, and so $h^* = h_0^* (j')^*$. Since $h_0^*$ is an isomorphism and $h^*$ is surjective, so is $(j')^*$. Thus reversing the argument of the previous paragraph, now applied to $S'$ instead of $S$, we see that $S'$ also separates $K$ and $K'$.
\end{proof}

We close this section with a characterization of when a toroidal set can be realized as an attractor for a flow, rather than a homeomorphism. The definition of attractor in the continuous case is a direct translation of that given for homeomorphisms in page \pageref{def:attract}.


We first recall a definition (see for instance \cite{chedwards1}). Let $T$ and $T'$ be two solid tori such that $T'$ is contained in the interior of $T$. The two tori are said to be \emph{concentric} if $T - {\rm int}\ T'$ is homeomorphic to the Cartesian product of a $2$--torus and the closed interval $[0,1]$. We will make use of the following result of Edwards \cite[Theorem 11, p. 13]{chedwards1}. Suppose that $A$, $B$ and $C$ are solid tori such that $A \subseteq {\rm int}\ B$ and $B \subseteq {\rm int}\ C$, with $\partial B$ bicollared. Then $A$ and $C$ are concentric if and only if $B$ is concentric with both $B$ and $C$.

\begin{theorem} \label{teo:realize_flow} Let $K \subseteq \mathbb{R}^3$ be a toroidal set. The following are equivalent:
\begin{itemize}
	\item[(i)] $K$ can be realized as an attractor for a flow.
	\item[(ii)] For every nested neighbourhood basis $(T_i)$ of $K$ comprised of polyhedral solid tori, all the $T_i$ are concentric for big enough $i$.
	\item[(iii)] There exists a nested neighbourhood basis $(T_i)$ of $K$ comprised of polyhedral solid tori such that all the $T_i$ are concentric for big enough $i$.
\end{itemize}
\end{theorem}
\begin{proof}

(i) $\Rightarrow$ (ii) Denote by $\varphi : \mathbb{R}^3 \times \mathbb{R} \longrightarrow \mathbb{R}^3$ the flow that realizes $K$ as an attractor and let $\mathcal{A}(K)$ be the basin of attraction of $K$. Let $L : \mathcal{A}(K) \longrightarrow [0,+\infty)$ be a proper Lyapunov function for $K$. This means that $L$ is a continuous function such that: (a) $L$ is strictly decreasing along the trajectories of $\varphi$ in $\mathcal{A}(K) - K$, (b) $L^{-1}(0) = K$ and (c) $L^{-1}[0,c]$ is compact for every $c > 0$. Let $P := L^{-1}[0,1]$.
\smallskip

{\it Claim.} $P$ is a tame solid torus.
\smallskip

{\it Proof of claim.} This is essentially stated by Chewning and Owen in \cite[Theorem 6, p. 76]{chewningowen1}. However, since the details of the proof are not given there, we provide them for completeness. From (a), (b) and (c) above it is easy to see that $P$ is a compact, positively invariant neighbourhood of $K$ and its topological frontier is ${\rm fr}\ P = L^{-1}(1)$. Moreover, ${\rm fr}\ P$ is a compact $2$--manifold. This nontrivial assertion is a consequence of the theory of generalized manifolds, and we refer the reader to the paper by Chewning and Owen cited earlier for a detailed proof and references. Using (a) and (c) it is straightforward to check that $\varphi$ provides an embedding of $({\rm fr}\ P) \times [0,1)$ onto an open neighbourhood $W$ of ${\rm fr}\ P$ in $P$. Since $({\rm fr}\ P) \times [0,1)$ is a $3$--manifold with boundary, so is $W$. Thinking of $P$ as the union of its topological interior (which is a $3$--manifold because it is open in $\mathbb{R}^3$) and $W$ we see that $P$ itself is a compact $3$--manifold whose boundary is ${\rm fr}\ P$. We shall therefore change our notation and write $\partial P$ instead of ${\rm fr}\ P$. Notice similarly that $\varphi$ also provides a embedding of $({\rm fr}\ P) \times [-1,1]$ onto a neighbourhood of $\partial P$, so the latter surface is bicollared in $\mathbb{R}^3$. This implies that it is a tame surface (see for instance \cite[Lemma 4, p. 13]{chedwards1}) and also that $P$ is a tame subset of $\mathbb{R}^3$ since any homeomorphism of $\mathbb{R}^3$ that sends $\partial P$ onto a polyhedron will also send the whole of $P$ onto a polyhedron.

By \cite[Theorem~3.6, pg. 233]{kapitanski1} the inclusion $K \subseteq P$ induces isomorphisms in \v{C}ech cohomology; therefore, $\check{H}^1(K) = \check{H}^1(P) = H^1(P)$. The cohomology of a compact manifold is finitely generated; thus, $\check{H}^1(K)$ is finitely generated and so it must be either $0$ or $\mathbb{Z}$ since $K$ is a toroidal set. If it were $0$, then also $H^1(P) = 0$ and a straightforward computation using Lefschetz duality (and $H^0(P) = \check{H}^0(P) = \mathbb{Z}$ and $H^q(P) = \check{H}^q(K) = 0$ for every $q \geq 2$) would yield that $\partial P$ has the homology of the $2$--sphere. Thus $\partial P$ would be the $2$--sphere and by the Sch\"onflies theorem for bicollared spheres (a deep theorem due to Brown \cite{brown1}) $P$ would be a $3$--cell and $K$ would be cellular, contradicting the assumption. (More details of this somewhat hasty argument can be found in the proof of \cite[Theorems 3 and 4, pp. 74 and 75]{chewningowen1}).

We are left with the only possibility that $H^1(P) = \mathbb{Z}$. Again using Lefschetz duality this implies that $\partial P$ has the homology of a $2$--torus; hence it is a $2$--torus. Let $T$ be a solid torus that is a neighbourhood of $K$ contained in $P$, which exists because $K$ is toroidal. Using the fact that $K$ attracts compact subsets of its basin of attraction, choose $\tau > 0$ so big that $P' := \varphi(P,\tau)$ is contained in $T$. Notice that the flow $\varphi$, acting on the time interval $[0,\tau]$, provides a strong deformation retraction of $P$ onto $P'$; thus, the inclusion $P' \subseteq P$ is a homotopy equivalence. In particular, it induces an isomorphism $\pi_1(P') \longrightarrow \pi_1(P)$. However, this homomorphism factors through $\pi_1(T) \cong \mathbb{Z}$, so it follows that $\pi_1(P)$ is a subgroup of $\mathbb{Z}$; hence it must be $\pi_1(P) \cong \mathbb{Z}$ (otherwise $\pi_1(P) = 0$, which can be ruled out arguing as in the previous paragraph). We thus have that $P$, which is the closure of one of the complementary domains of the $2$--torus $\partial P$, has $\pi_1 \cong \mathbb{Z}$. Invoking again the theorem of Alexander \cite[Theorem 1, p. 107]{rolfsen1} already used in the proof of Theorem \ref{teo:separate} we conclude that $P$ must be a solid torus. (Let us mention that the theorem of Alexander applies to polyhedral $2$--tori, but since $\partial P$ is tame and can therefore assumed to be polyhedral after performing an ambient homeomorphism, it also applies to $\partial P$). This concludes the proof of the claim. $_{\blacksquare}$
\smallskip

Now let $(T_i)$ be a nested neighbourhood basis of $K$ comprised of polyhedral solid tori. Let $i_0$ be big enough so that $T_i \subseteq {\rm int}\ P$ for every $i \geq i_0$. We claim that any such $T_i$ is concentric with $P$, and this in turn implies that any two $T_i$ and $T_j$ with $i,j \geq i_0$ are concentric, which is what we need to prove. Let, then, $T_i$ be contained in ${\rm int}\ P$. Choose $\tau > 0$ so big that $P' := \varphi(P,\tau)$ is contained in the interior of $T_i$. The flow itself provides a homeomorphism from $(\partial P) \times [0,\tau]$ onto $P - {\rm int}\ P'$, so $P$ and $P'$ are concentric. By the results of Edwards \cite[Theorem 11, p. 13]{chedwards1} mentioned before the statemente of the theorem, $T_i$ is also concentric with $P$.

(ii) $\Rightarrow$ (iii) This is obvious.

(iii) $\Rightarrow$ (i) Let $(T_i)$ be a nested neighbourhood basis of $K$ as described in (iii) and assume, possibly after discarding the first few elements of $(T_i)$, that all the $T_i$ are concentric. The region $R_i := T_i - {\rm int}\ T_{i+1}$ comprised between each $T_i$ and the next $T_{i+1}$ is then homeomorphic to the Cartesian product of a $2$--torus $\Sigma$ and a closed interval; that is, there exist homeomorphisms $h_i : R_i \longrightarrow \Sigma \times [i,i+1]$. Observe that the boundary of $R_i$ is the disjoint union of $\partial T_i$ and $\partial T_{i+1}$, and we may assume without loss of generality that these are carried by $h_i$ onto correspond to $\Sigma \times \{i\}$ and $\Sigma \times \{i+1\}$ respectively. We would like to paste all the $h_i$ together to produce a homeomorphism between $\bigcup R_i = T_1 - K$ onto $\Sigma \times [1,+\infty)$, because in that case $T_1 - K$ can be compactified to the $3$--manifold $S \times [0,+\infty]$ and this is sufficient to guarantee that $K$ can be realized as an attractor for a flow (see \cite[Proposition 26, p. 6175]{mio5}). The minor technical problem arises that two consecutive $h_i$ and $h_{i+1}$ need not coincide on their common domain $R_i \cap R_{i+1} = \partial T_{i+1}$, but this can be easily fixed by replacing each $h_i$ with another $h'_i$ inductively as follows. First set $h'_1 := h_1$. Now consider the restrictions of $h'_1$ and $h_2$ to $R_1 \cap R_2$: both are homeomorphisms onto $\Sigma \times \{2\}$, and so we can define a homeomorphism $t$ of $\Sigma$ as \[t : \Sigma \ni s \longmapsto \text{first coordinate of\ } (h'_1 \circ h_2^{-1}(s,2)) \in \Sigma.\] Set $h'_2 := (t \times {\rm id}) \circ h_2$. This is a new homeomorphism of $R_2$ onto $\Sigma \times [2,3]$ that, now, coincides with $h_1$ on $R_1 \cap R_2$. Indeed, for any $p \in R_1 \cap R_2$ we have $h_2(p) = (s,2)$ for some $s \in \Sigma$ and \[h'_2(p) = (t \times {\rm id}) \circ h_2(p) = (t \times {\rm id}) (s,2) = ((\text{first coordinate of } h'_1(p),2) = h'_1(p).\] Replacing $h_2$ with $h'_2$ and performing the same construction we may obtain $h'_3, h'_4, \ldots$ with the required property that they agree on $R_i \cap R_{i+1}$ and therefore, taken together, define a homeomorphism between $\bigcup R_i$ and $T_1 - K$, as we needed.
\end{proof}

\begin{example} The Whitehead continuum and generalized solenoids do not satisfy condition (iii) of the preceding theorem; thus, none of them can be realized as an attractor for a flow. This can also be proved using different techniques: the case of solenoids is an immediate consequence of the theorem that attractors for flows must have the Borsuk shape of a finite polyhedron (\cite{gunthersegal1}, \cite{sanjurjo1}) and both solenoids and the Whitehead continuum were already discussed using yet another approach in \cite[Examples 46 and 47, p. 3622]{mio6}.
\end{example}

\begin{remark} It follows from the proof of the preceding theorem that a toroidal set $K$ that is an attractor for a flow must have $\check{H}^1(K;\mathbb{Z}) = \mathbb{Z}$. This homological obstruction, which is simpler but weaker than the condition that $K$ has a neighbourhood basis of concentric solid tori, is in fact enough to prove that neither the Whitehead continuum nor a generalized solenoid can be realized as attractors for a flow, since they have zero and non finitely generated $\check{H}^1(\, \cdot\, ; \mathbb{Z})$ respectively. However, if one modifies the definition of the Whitehead continuum by having, in the notation of Example \ref{ex:whitehead}, $h(T_0)$ wind an additional time inside $T_0$, then the resulting toroidal set has $\check{H}^1( \, \cdot \, ; \mathbb{Z}) = \mathbb{Z}$ so there is no homological obstruction to it being realizable as an attractor for a flow. Theorem \ref{teo:realize_flow}, however, still shows that in fact it cannot.
\end{remark}

\section{A comparison with previous techniques}

In a previous paper \cite{mio5} one of the authors associated to every continuum $K \subseteq \mathbb{R}^3$ a number $r(K) = 0,1,2,\ldots,\infty$ which, in some sense, measures how wildly $K$ sits in $\mathbb{R}^3$. This number was shown to be finite for attractors and then used to exhibit examples of continua $K$ that cannot be realized as attractors because they have $r(K) = \infty$. In this final section we recall the definition of $r(K)$ and consider it within the realm of toroidal sets.

Let $K \subseteq \mathbb{R}^3$ be a compact set. Tiling $\mathbb{R}^3$ with small cubes it is easy to see that $K$ has a neighbourhood basis all of whose elements are compact $3$--manifolds which can be assumed to be polyhedral. Given a number $r = 0,1,2,\ldots$ consider the following property $(P_r)$ that $K$ may or may not have: $K$ has a neighbourhood basis of compact, polyhedral $3$--manifolds $N_k$ such that ${\rm dim}\ H_1(N_k;\mathbb{Z}_2) = r$. Then set $r(K)$ to be the minimum $r$ such that $(P_r)$ holds or $r(K) = +\infty$ if $K$ does not have the property $(P_r)$ for any $r$; that is, if for any neighbourhood basis of compact, polyhedral $3$--manifolds $N_k$ one has ${\rm dim}\ H_1(N_k;\mathbb{Z}_2) \rightarrow +\infty$ as $k \rightarrow +\infty$. (The use of $\mathbb{Z}_2$ coefficients is not strictly necessary, but it is convenient for technical reasons. In particular, owing to the universal coefficient theorem, we may and will identify $H_1$ and $H^1$ in the sequel).

Because $r(K)$ is defined in terms of neighbourhoods of $K$, it does not seem unreasonable to believe that it may, to some extent, measure how $K$ sits in $\mathbb{R}^3$. And, indeed, it can be shown that $r(K)$ is bounded below by ${\rm dim}\ H_1(K;\mathbb{Z}_2)$ (which is a topological invariant of $K$; that is, it does not depend on its embedding in $\mathbb{R}^3$) and coincides with it when $K$ is a polyhedron, but it may be strictly bigger when $K$ is wildly embedded in $\mathbb{R}^3$. Clearly $r(K)$ is an invariant under ambient homeomorphisms: if $h$ is a homeomorphism of $\mathbb{R}^3$, then $r(K) = r(h(K))$. 

As for toroidal sets, we have the following remark:

\begin{remark} \label{rem:r_tor} If $K\subseteq\mathbb{R}^3$ is a toroidal set, then $r(K)=1$.
\end{remark}

We will use the following property of $r(K)$: if $\check{H}^2(K;\mathbb{Z}_2) = 0$ and $r(K) = 0$, then $K$ has a basis of neighbourhoods each of which is a disjoint union of finitely many polyhedral cells (see the proof of \cite[Theorem~16, p. 3601]{mio6}). It follows easily from this statement that $\check{H}^1(K;\mathbb{Z}) = 0$ and each connected component of $K$ is cellular. In particular, if $K$ is connected then it is cellular itself.

\begin{proof}[Proof of Remark \ref{rem:r_tor}] Since $K$ has a basis of neighbourhoods comprised of solid tori, clearly $r(K)\leq 1$. Suppose that $r(K) = 0$. We saw earlier that toroidal sets are connected and have vanishing \v{C}ech cohomology in degree $2$, so according to the previous paragraph, if $r(K) = 0$ then $K$ would be cellular. However, this is excluded by definition.
\end{proof}

It is easily shown that the invariant $r(K)$ is finite whenever $K$ is an attractor. Then, by proving that the Alexander horned sphere and some other compacta have an infinite $r(K)$, it was shown in \cite{mio5} that none of them can be realized as attractors. It is a consequence of the Remark \ref{rem:r_tor} that this approach is not sensitive enough to rule out a toroidal set $K$ as an attractor. This is to be expected, since $r(K)$ is an invariant of homological nature and, as it is well known, homology does not detect knottedness. The genus of a toroidal set, on the other hand, is designed to do precisely this and therefore provides a finer invariant (albeit only applicable to toroidal sets, whereas $r(K)$ can be computed for any compactum $K \subseteq \mathbb{R}^3$).

The following result is a partial converse to Remark \ref{rem:r_tor}:

\begin{theorem} \label{teo:rdeK} Let $K \subseteq \mathbb{R}^3$ be a continuum with $r(K) = 1$ and $\check{H}^2(K;\mathbb{Z}) = 0$. If $\check{H}^1(K;\mathbb{Z}) \neq 0, \mathbb{Z}$ then $K$ is a toroidal set.
\end{theorem}

To prove it we need an auxiliary lemma. Let $N \subseteq \mathbb{R}^3$ be a polyhedral, compact, connected $3$--manifold. Then $\overline{\mathbb{R}^3 - N}$ is a disjoint union of polyhedral $3$--manifolds, exactly one of which is unbounded while the others are bounded. Denote by $M$ be union of the latter and set $\hat{N} := M \cup N$. We say that the compact $3$--manifold $\hat{N}$ is the result of \emph{capping the holes} of $N$.

\begin{lemma} \label{lem:cap_holes} $H_2(\hat{N};\mathbb{Z}_2) = 0$ and $\dim\ H_1(\hat{N};\mathbb{Z}_2) \leq \dim\ H_1(N;\mathbb{Z}_2)$.
\end{lemma}
\begin{proof} Through this proof we use homology with $\mathbb{Z}_2$ coefficients without further mention. Let $U$ be one of the bounded connected components of $\mathbb{R}^3 - N$, so that it is also a connected component of $\mathbb{S}^3 - N$ and in particular $H_2(U)$ is a direct summand of $H_2(\mathbb{S}^3 - N)$. By Alexander duality $H_2(\mathbb{S}^3 - N) = \tilde{H}^0(N) = 0$, and in particular $H_2(U) = 0$ too. Consider the compact and connected $3$--manifold $\overline{U}$. It is homotopy equivalent to its interior, so $H_2(\overline{U}) = 0$ and by Lefschetz duality $H_1(\overline{U},\partial{\overline{U}}) = H_2(\overline{U}) = 0$. The exact sequence for the pair $(\overline{U},\partial \overline{U})$ in reduced homology \[\xymatrix{H_1(\partial \overline{U}) \ar[r] & H_1(\overline{U}) \ar[r] & H_1(\overline{U},\partial \overline{U}) = 0 \ar[r] & \tilde{H}_0(\partial \overline{U}) \ar[r] & \tilde{H}_0(\overline{U}) = 0}\] then implies that (i) the inclusion $\partial \overline{U} \subseteq \overline{U}$ induces a surjective map in one--dimensional homology and (ii) $\partial \overline{U}$ is connected.

Consider the compact $3$--manifold $N \cup \overline{U}$. Observe that $N \cap \overline{U} = \partial \overline{U}$ and $N \cup \overline{U}$ is connected because it is the union of two connected, nondisjoint sets. This leads to the Mayer--Vietoris sequence in reduced homology \[\xymatrix{H_1(\partial\overline{U}) \ar[r]^-{\alpha} & H_1(N) \oplus H_1(\overline{U}) \ar[r] &  H_1(N \cup \overline{U}) \ar[r] & \tilde{H}_0(\partial \overline{U}) = 0}\] We argued in the previous paragraph that the inclusion induced map $H_1(\partial \overline{U}) \longrightarrow H_1(\overline{U})$ is surjective; thus, $\dim\ {\rm im}(\alpha) \geq \dim\ H_1(\overline{U})$. Therefore, since the arrow entering $H_1(N \cup \overline{U})$ is also surjective, \[\dim\ H_1(N \cup \overline{U}) = \dim\ H_1(N) + \dim\ H_1(\overline{U}) - \dim\ {\rm im}(\alpha) \leq \dim\ H_1(N).\]

Performing the above construction repeatedly to adjoin all the bounded connected components of $\mathbb{R}^3 - N$ to $N$ we obtain a compact, connected manifold $\hat{N}$ such that $\dim\ H_1(\hat{N}) \leq \dim\ H_1(N)$ and $\mathbb{R}^3 - \hat{N}$ has a single connected (unbounded) component. Then by Alexander duality again $H_2(\hat{N}) = \tilde{H}_0(\mathbb{R}^3 - \hat{N}) = 0$. This completes the proof.
\end{proof}

The proof of the next lemma is essentially the same as that of \cite[Lemma 20, p. 3603]{mio6}.

\begin{lemma} \label{lem:contained} Let $N$ be a compact polyhedral $3$--manifold in $\mathbb{R}^3$. Let $S \subseteq {\rm int}\ N$ be a closed, connected, polyhedral surface such that its fundamental class $[S] \in H_2(S;\mathbb{Z}_2)$ is zero in $H_2(N;\mathbb{Z}_2)$. Then the bounded connected component of $\mathbb{R}^3 - S$ is contained in $N$.
\end{lemma}

\begin{proof}[Proof of Theorem \ref{teo:rdeK}] Since $K$ is connected and has $r(K) = 1$, it follows that $K$ has a nested neighbourhood basis of polyhedral, compact $3$--manifolds $N_i$ that are connected and have ${\rm dim}\ H_1(N_i;\mathbb{Z}_2) = 1$. Consider the collection $\hat{N}_i$ obtained from the $N_i$ by capping their holes in the way described just before Lemma \ref{lem:cap_holes}.
\smallskip

{\it Claim.} $(\hat{N}_i)$ is still a neighbourhood basis of $K$.

{\it Proof of claim.} Since $\check{H}^2(K;\mathbb{Z}_2)$ is the direct limit of $H^2(N_i;\mathbb{Z}_2)$ and $\check{H}^2(K;\mathbb{Z}_2) = 0$ by assumption, for each $i$ there exists $j > i$ such that the inclusion $N_j \subseteq N_i$ induces the zero map in $H^2(\, \cdot\, ; \mathbb{Z}_2)$ and therefore also in $H_2(\, \cdot\, ; \mathbb{Z}_2)$. Consider $\hat{N}_j$. Since $H_2(\hat{N}_j;\mathbb{Z}_2) = 0$, it is easy to see that $\partial \hat{N}_j$ is connected; in fact, it is precisely one of the connected components of $\partial N_j$. Thus in particular its fundamental class $[\partial \hat{N}_j]$ is sent to zero in $H_2(N_i;\mathbb{Z}_2)$ and so by Lemma \ref{lem:contained} we have $\hat{N}_j \subseteq N_i$. This proves that the $(\hat{N}_i)$ are indeed a neighbourhood basis of $K$. $_{\blacksquare}$
\smallskip

By replacing the original $(N_i)$ with the capped $(\hat{N}_i)$ and denoting the latter again by $(N_i)$ for simplicity we see (invoking Lemma \ref{lem:cap_holes}) that $K$ has a nested neighbourhood basis of compact, polyhedral, connected $3$--manifolds such that $\dim\ H_1(N_i;\mathbb{Z}_2) = 1$ and $H_2(N_i;\mathbb{Z}_2) = 0$. A straightforward computation using Lefschetz duality and the long exact sequence for the pair $(N_i,\partial N_i)$ shows that $\partial N_i$ has the homology (with $\mathbb{Z}_2$ coefficients) of a $2$--torus; hence it is a $2$--torus by the classification of connected closed surfaces. By the theorem of Alexander either $N_i$ is a solid torus or $\overline{\mathbb{S}^3 - N_i}$ is a solid torus (or both are). If the first case happens for infinitely many $i$ then $K$ is toroidal by definition and we are finished. The other case requires some more argumentation. Let us assume, then, that $\overline{\mathbb{S}^3 - N_i}$ is a solid torus $T_i$ for every $i$. Observe that now we have an ascending sequence of solid tori; that is, $T_1 \subseteq T_2 \subseteq \ldots$, and their union is precisely $\mathbb{S}^3 - K$. Denote by $w_i$ the winding number of $T_i$ inside $T_{i+1}$. If we had $w_i = 0$ for infinitely many $i$ then $H_1(\mathbb{S}^3 - K;\mathbb{Z}) = 0$ because this homology group is the direct limit of the direct sequence $H_1(T_i;\mathbb{Z}) \longrightarrow H_1(T_{i+1};\mathbb{Z})$. By Alexander duality this implies $\check{H}^1(K;\mathbb{Z}) = 0$ contradicting the hypotheses of the theorem. Thus we may assume that $w_i \geq 1$ for every $i$. By Equation \eqref{eq:schubert} we then have $g(T_i) \geq w_i \cdot g(T_{i+1}) \geq g(T_{i+1})$ and so the sequence of genera $g(T_i)$ is nonincreasing so it must eventually stabilize: $g(T_i) = g$ for every $i \geq i_0$. Suppose that $g > 0$. From $g \geq w_i \cdot g$ we conclude that $w_i \leq 1$ for every $i \geq i_0$; since $w_i \neq 0$ as argued above, $w_i = 1$ for every $i \geq i_0$. This implies that $H_1(\mathbb{S}^3 - K;\mathbb{Z}) = \mathbb{Z}$ and therefore $\check{H}^1(K;\mathbb{Z}) = \mathbb{Z}$ by Alexander duality, contrary to the assumption of the theorem. Therefore $g = 0$ and so $T_i$ is unknotted for $i \geq i_0$, which implies that $N_i$ is an (unknotted) solid torus and so $K$ is an unknotted toroidal set.
\end{proof}

The following corollary generalizes the above result to the non connected case:

\begin{corollary} \label{cor:rdeK} Let $K \subseteq \mathbb{R}^3$ be a compactum with $r(K) = 1$ and $\check{H}^2(K;\mathbb{Z}) = 0$. If $\check{H}^1(K;\mathbb{Z}) \neq 0, \mathbb{Z}$ then exactly one of the connected components of $K$ is a toroidal set and the remaining ones are cellular sets.
\end{corollary}
\begin{proof} We shall write the proof as if $K$ had infinitely many connected components. The same argument works (but can be more easily written) if $K$ has only finitely many connected components.

Let $(N_i)$ be a neighbourhood basis of $K$ such that each $N_i$ is a compact manifold and $N_{i+1} \subseteq N_i$. For each $i$ let $N_{i1}, \ldots, N_{im_i}$ denote the (finitely many) connected components of $N_i$, and set $K_{ij} := N_{ij} \cap K$. 

Now, for fixed $i$ we have $K = \uplus_{j=1}^{m_i} K_{ij}$. Since $\check{H}^2(K;\mathbb{Z}) = \oplus_j \check{H}^2(K_{ij};\mathbb{Z})$ the condition that $\check{H}^2(K;\mathbb{Z}) = 0$ implies that the same is true of each $K_{ij}$. Also, it is clear from the definition of $r$ that $r(K) = \sum_j r(K_{ij})$; thus, there is precisely one $0 \leq j_0 \leq m$ such that $r(K_{i{j_0}}) = 1$ whereas $r(K_{ij}) = 0$ for the remaining $j \neq j_0$. Without loss of generality we shall assume that $j_0 = 1$. Then for every $1 \leq j \leq m_i$ the nullity property of $r$ implies that $\check{H}^1(K_{ij};\mathbb{Z}) = 0$ and every connected component $C$ of $K$ contained in some $K_{ij}$ other than $K_{i1}$ is cellular. Also, since $\check{H}^1(K;\mathbb{Z}) = \oplus_j \check{H}^1(K_{ij};\mathbb{Z}) = \check{H}^1(K_{i1};\mathbb{Z})$, it follows that the inclusion $K_{i1} \subseteq K$ induces an isomorphism $\check{H}^1(K;\mathbb{Z}) = \check{H}^1(K_{i1};\mathbb{Z})$, so in particular $\check{H}^1(K_{i1};\mathbb{Z}) \neq 0,\mathbb{Z}$.

Now let $i \rightarrow +\infty$. Since we have $N_{i+1} \subseteq N_i$ for every $i$, the $K_{i1}$ form a decreasing sequence. Let $K'$ be the intersection of this decreasing sequence. It is straightforward to check that $K'$ is connected, has $r(K') \leq 1$, and (by the continuity property of \v{C}ech cohomology) also $\check{H}^2(K';\mathbb{Z}) = 0$. Similarly, because both inclusions $K_{i1} \subseteq K$ and $K_{{i+1}0} \subseteq K$ induce isomorphisms in $\check{H}^1$ for each $i$, the same is true of the inclusion $K_{{i+1}0} \subseteq K_{i0}$ and so also of $K' \subseteq K$. In particular we cannot have $r(K') = 0$ because then $K'$ would be cellular and $\check{H}^1(K';\mathbb{Z}) = 0$, which is not the case. Thus $r(K') = 1$. Since $K'$ is connected, Theorem \ref{teo:rdeK} applies and shows that $K'$ is a toroidal set. Any other component $C$ of $K$ is contained, for some big enough $i_0$ onwards, on some $K_{ij}$ with $j \geq 1$; thus, it is cellular.
\end{proof}

\bibliographystyle{plain}
\bibliography{biblio}

\end{document}